\documentclass{amsart}

\usepackage[a4paper,top=4cm,bottom=4cm,left=2.5cm,right=2.5cm]{geometry}
\usepackage[utf8]{inputenc}

\usepackage{amssymb, amsmath, amsthm, amsfonts, euscript, color}
\usepackage{tabularx}
\usepackage{array}
\usepackage{graphicx} 
\usepackage{comment}
\usepackage[hypertexnames=false,
    backref=page,
    pdftex,
    pdfpagemode=UseNone,
    breaklinks=true,
    extension=pdf,
    colorlinks=true,
    linkcolor=blue,
    citecolor=blue,
    urlcolor=blue,
]{hyperref}
\usepackage{amsrefs}

\usepackage{enumitem}

\usepackage{verbatim}

\theoremstyle{plain}
\newtheorem{theorem}{Theorem}[section]
\newtheorem{lemma}[theorem]{Lemma}
\newtheorem{proposition}[theorem]{Proposition}
\newtheorem{corollary}[theorem]{Corollary}

\newtheorem*{theorem*}{Theorem}

\theoremstyle{definition}
\newtheorem{definition}[theorem]{Definition}

\newtheorem{example}[theorem]{Example}
\newtheorem*{corollary*}{Corollary}
\newtheorem*{example*}{Example}
\newtheorem*{definition*}{Definition}

\theoremstyle{remark}
\newtheorem{remark}[theorem]{Remark}

\newcommand{\Z}{\mathbb{Z}} 
\newcommand{\Q}{\mathbb{Q}} 
\newcommand{\R}{\mathbb{R}} 
\newcommand{\C}{\mathbb{C}} 
\newcommand{\del}{\partial}
\newcommand{\delbar}{\overline{\partial}}
\newcommand{\ddbar}{\partial\overline{\partial}}
\DeclareMathSymbol{\Finv} {\mathord}{AMSb}{"60}

\title[Holomorphically parallelizable solvmanifolds, special metrics and deformations]{Holomorphically parallelizable solvmanifolds with special metrics and their deformations}
\author{Ettore Lo Giudice, Lapo Rubini and Adriano Tomassini}
\address[Ettore Lo Giudice]{
Dipartimento di Scienze Matematiche, Fisiche e Informatiche
Unit\`a di Matematica e Informatica\\
Universit\`a degli Studi di Parma\\
Parco Area delle Scienze 53/A, 43124\\
Parma, Italy \& \\
Dipartimento di Matematica e Informatica \\
Universita degli Studi di Ferrara \\
Via Machiavelli 35, 44121 Ferrara, Italy}
\email{ettore.logiudice@unipr.it, ettore.logiudice@unife.it}

\address[Lapo Rubini]{Dipartimento Di Matematica E Informatica “Ulisse Dini”, Università Degli Studi di Firenze, viale Morgagni 67/a, 50134, Florence, Italy}
\email{lapo.rubini@unifi.it}

\address[Adriano Tomassini]{
Dipartimento di Scienze Matematiche, Fisiche e Informatiche
Unit\`a di Matematica e Informatica\\
Universit\`a degli Studi di Parma\\
Parco Area delle Scienze 53/A, 43124\\
Parma, Italy}
\email{adriano.tomassini@unipr.it}

\thanks{The second and the third authors have been supported by the PRIN 2022 project ``Real and Complex Manifolds: Geometry and holomorphic dynamics" (code 2022AP8HZ9). The authors are partially supported by GNSAGA of INdAM}

\begin{document}

\begin{abstract}
    We investigate the existence of strong K\"ahler with torsion metrics along deformations of the Iwasawa manifold and of the holomorphically parallelizable Nakamura manifold. We also show that the class of deformations of the holomorphically parallelizable Nakamura manifold yielding a non-left-invariant complex structure admits a balanced metric but does not admit any strong K\"ahler with torsion metric. We then construct the Kuranishi space of a $4$-dimensional holomorphically parallelizable solvmanifold and study whether small deformations of such a manifold admit SKT metrics. Finally, we provide some results concerning the existence of metrics satisfying $\del \delbar \omega = 0$, $\del \delbar \omega^{2} = 0$ on a particular class of $2$-step nilpotent nilmanifolds.
\end{abstract} 
\maketitle
\tableofcontents

\tableofcontents

\section{Introduction}\label{Section 1}

Many compact complex manifolds that do not admit any K\"ahler metric can often be equipped with Hermitian metrics that satisfy weaker differential conditions. The existence of such metrics on a complex manifold yields important analytic and geometric properties. In recent years, alongside the study of the existence of special Hermitian metrics, considerable attention has been paid to the coexistence of different special metrics on the same underlying complex manifolds. 

In the present paper, we mainly focus on three classes of Hermitian metrics: {\em balanced} (also called {\em semi-K\"ahler}) (see \cites{Gauduchon77-fibres,Gauduchon84, GrayHervella80,Michelsohn82}), {\em Strong K\"ahler with torsion} (shortly SKT, also called {\em pluriclosed}) (see \cite{Bismut89}) and {\em astheno-K\"ahler} (see \cite{JostYau93}). 

If $(M^{n},J,g)$ is a Hermitian manifold and we denote by $\omega$ the fundamental form of $g$, then from \cite{Gauduchon97} there exists a one-parameter family of Hermitian connections 
\[
\nabla^{t} = t \nabla^{C} + (1-t) \nabla^{0},
\]
where $\nabla^{C}$ denotes the {\em Chern connection} and $\nabla^{0}$ the {\em first canonical connection}. For $t=-1$, $\nabla^t$ coincides with the so called {\em Bismut connection}, which is the unique Hermitian connection whose torsion tensor is totally skew-symmetric. In this framework, balanced metrics are characterized by the condition $d \omega^{n-1} = 0$, which is equivalent to the Chern connection of $g$ having traceless torsion (see \cite{Michelsohn82}). SKT or pluriclosed metrics are those for which the torsion tensor of the Bismut connection is $d$-closed, a condition equivalently expressed by $\del \delbar \omega = 0$. Finally, astheno-K\"ahler metrics are defined by the equation $\del \delbar \omega^{n-2} = 0$.

These special Hermitian metrics arise naturally in the setting of solvmanifolds and nilmanifolds. Throughout this paper, by a {\em solvmanifold} (resp. {\em nilmanifold}) we mean the compact quotient of a connected, simply connected solvable (resp. nilpotent) Lie group by a lattice, equipped with an integrable almost complex structure. It was proved by E. Abbena and A. Grassi (see \cite{AbbenaGrassi86}) that a complex Lie group admits a left-invariant balanced metric if and only if it is unimodular. Consequently, by a result of H.-C. Wang (see \cite{Wang54}), every {\em holomorphically parallelizable manifold}, i.e., complex manifolds $M$ with holomorphically trivial $T^{1,0}M$, is balanced.

\medskip

The coexistence of different special Hermitian metrics on the same underlying compact complex manifold has been investigated by several authors. More precisely, it is known that a non-K\"ahler compact complex manifold cannot admit a metric that is simultaneously SKT and balanced, nor one that is simultaneously astheno-K\"ahler and balanced (see \cites{AlexandrovIvanov2001, MatsuoTakahashi2001}). Furthermore, in \cite[Theorem 1.1]{FinoVezzoni2016} A. Fino and L. Vezzoni established that a nilmanifold equipped with a left-invariant complex structure cannot admit two different metrics, one SKT and the other balanced, unless the manifold is K\"ahler. They also conjectured that this statement holds for every compact complex manifold (see \cite[Problem 3]{FinoVezzoni15}). We also recall that, in \cite[Corollary 3.2]{FinoKasuyaVezzoni15} the authors proved that if the complex structure is left-invariant, there are no SKT metrics on holomorphically parallelizable solvmanifolds which are not nilmanifolds.

On the other hand, despite the fact that the same metric cannot be both astheno-K\"ahler and balanced unless it is K\"ahler, several examples of compact complex manifolds admitting two different metrics, one astheno-K\"ahler and the other balanced, have been provided in the literature (see \cites{FinoGrantcharovVezzoni19,LatorreUgarte2017,SferruzzaTomassini23}).

\medskip

A complementary direction in the study of special Hermitian metrics concerns their behavior under {\em blow-ups} and {\em modifications}. It is proved by L. Alessandrini and G. Bassanelli that the modification of a balanced manifold remains balanced (see \cite{AlessandriniBassanelli91modifications}). Moreover, it was shown in \cite[Proposition 3.2]{FinoTomassini09} that the SKT condition is stable under blow-ups along a compact complex submanifold. 

In \cite{FinoTomassini11}, the authors introduced another class of Hermitian metrics which are stable under blow-ups along a compact complex submanifold (see \cite[Proposition 2.4]{FinoTomassini11}). These metrics are defined by the following condition:
\begin{equation}\label{FT metrics}
    \del \delbar \omega = 0, \quad  \del \delbar \omega^{2} = 0.
\end{equation}
As shown in \cite{FinoTomassini11}, condition \eqref{FT metrics} forces the metric to satisfy $\del \delbar \omega^{k} = 0$, for all $k \ge 1$. Moreover, an important reason to study these metrics is that they leave the total Monge-Amp\'ere volume unchanged (see \cites{AngellaGuedjLu,GuanLi}).

\medskip

The aim of this paper is to study the existence of metrics satisfying condition \eqref{FT metrics} on a particular class of nilmanifolds equipped with a left-invariant complex structure as well as to study deformations of compact complex manifolds admitting special Hermitian metrics. 

To address the former topic, we study metrics satisfying condition \eqref{FT metrics} on a class of nilmanifolds already studied in the literature (see for example \cites{FinoTomassini11,SferruzzaTomassini23}). Namely, we study the class of nilmanifolds equipped with a left-invariant complex structure that admits a $(1,0)$-coframe satisfying 
\begin{equation}\label{nilpotent Lie algebra T}
    \begin{cases}
        d \varphi^{j} = 0 , \quad  j=1, \dots, n-1, \\
        d \varphi^{n} \in \Lambda^{2} \langle \varphi^{1}, \dots, \varphi^{n-1}, \varphi^{\overline{1}}, \dots, \varphi^{\overline{n-1}} \rangle.
    \end{cases}
\end{equation}
In complex dimension $3$, A. Fino, M. Parton and S. Salamon, in \cite[Theorem 1.2]{FinoPartonSalamon04}, established a fundamental result which completely characterizes the existence of left-invariant SKT metrics on nilmanifolds equipped with a left-invariant complex structure (which coincides with astheno-K\"ahler metrics for dimensional reasons). More precisely, the authors show that either every left-invariant Hermitian metric is SKT or none is. In higher dimensions, we prove that for a nilmanifolds satisfying \eqref{nilpotent Lie algebra T} either every metric satisfies condition \eqref{FT metrics} or none does, proving the following result.

\begin{theorem}[Lemma \ref{tutte skt}, Theorem \ref{Tutte FT}, ]\label{teorema intro T}
    Let $(M\doteq\Gamma\backslash G, J)$ be a nilmanifold of complex dimension $n$ admitting a $(1,0)$-coframe $\{\varphi^1,...,\varphi^n\}$ that satisfies $\eqref{nilpotent Lie algebra T}$. 
    Then, none left-invariant Hermitian metric is SKT or every metric satisfies condition \eqref{FT metrics}. 
\end{theorem}
Examples of compact complex manifolds that admit metrics satisfying condition \eqref{FT metrics} were constructed in \cite[Theorem 2.7]{FinoTomassini11} and \cite[Theorem 5.6]{CiulicaOtimanStanciu}. In the former, examples of $4$-dimensional nilmanifolds admitting such metrics are provided, while in the latter the authors prove the existence of such metrics on a particular class of Endo-Pajitnov manifolds. Here, we construct new examples of nilmanifolds that admit such metrics by using the deformations studied in \cite[Example 1]{PiovaniSferruzza21} (see Example \ref{Condition 1 preservata}).

We recall that, in terms of both finding new examples of manifolds carrying special structures and studying the properties of such structures, {\em deformation theory} (\cite{Kodaira05}) provides another important tool. Several relevant stability results have been established in the literature. Among them, a fundamental result is the stability of the K\"ahler condition, proved in their celebrated paper by K. Kodaira and D. C. Spencer (see \cite{KodairaSpencer60}). Other important results are the stability of the $\del\delbar$-Lemma (e.g. \cites{AngellaTomassini13,Wu06}), and the stability of generalized K\"ahler structures with one pure spinor
under small deformations of generalized complex structures (see \cite{GotoGeneralizedKahler}). Although analogous statements are generally not valid for special Hermitian metrics (see \cites{AlessandriniBassanelli90, FinoTomassini09}), positive results have nevertheless been obtained, concerning either necessary or sufficient conditions (see \cites{RaoWanZhao21, RaoWanZhao18, Sferruzza22, Sferruzza23, PiovaniSferruzza21, AngellaUgarte17, Ciulica25}). It is worth mentioning that an important class of special Hermitian metrics which is stable under small deformation of the complex structure is the class of {\em strongly Gauduchon metrics} (see \cite[Proposition 4.1]{Xiao15}). We recall that a Hermitian metric $g$ is called strongly Gauduchon if its fundamental form $\omega$ is such that $\del \omega^{n-1}$ is $\delbar$-exact (see \cite{Popovici13}). 

\medskip

Another important goal of the paper is to study the coexistence of different special Hermitian metrics along the deformation of a compact complex manifolds. We recall that, in \cite{Nakamura75}, I. Nakamura studied the small deformations of some well known three dimensional manifolds, namely the {\em Iwasawa manifold} and the {\em holomorphically parallelizable Nakamura manifold}. Concerning the Iwasawa manifold, thanks also to \cites{Angella13, AngellaTomassini11}, we investigate in Section \ref{subsection Iwasawa} the SKT condition on its deformations. Note that in \cite{AngellaFranziniRossi}*{Theorem 4.1} all $6$-dimensional nilmanifolds admitting SKT metrics are classified. A direct application of \cite[Theorem 1.2]{FinoPartonSalamon04} shows that sufficiently small deformations of the Iwasawa manifold cannot admit any SKT metric and that there is a deformation connecting the standard complex structure on the Iwasawa manifold to a circle of complex structures that admit an SKT metric. We prove the following.

\begin{theorem}[Theorem \ref{no SKT intorno a Iwasawa}, Theorem \ref{curva balanced SKT}]\label{Deformazioni Iwasawa - introduzione} 
    Sufficiently small deformations of the Iwasawa manifold provided in \cite[p. 95]{Nakamura75} cannot admit SKT metrics. Moreover, the deformation with $t_{11}=t_{22}=0$ and $t_{12}=-t_{21}$ connects the standard complex structure on the Iwasawa manifold to a circle of complex structures that admits SKT metrics, defined by $|t_{21}|^{2} = 2-\sqrt{3}$. For $|t_{21}|^{2}  = 2+\sqrt{3}$, we have another circle of complex structures that admit SKT metrics. 
\end{theorem}

In Section \ref{subsection Nakamura}, we study the SKT condition on some of the deformations of the holomorphically parallelizable Nakamura manifold \cite[III-(3b)]{Nakamura75}. We prove that small deformations belonging to classes $(1)$, $(3)$ and $(4)$ - following Nakamura's classification - admit balanced metrics but do not admit any SKT metric. Moreover, classes $(3)$ and $(4)$ provide examples of solvmanifolds equipped with non-invariant complex structures that admit balanced metrics but do not carry any SKT metric, consistent with the Fino-Vezzoni conjecture (see Theorem \ref{deformation of Nakamura}).

\medskip

To investigate the existence of special metrics along deformations of a complex structure, in Section \ref{Section Kuranishi} we compute the Kuranishi space of the $4$-dimensional solvmanifold studied in \cite[Section 6]{SferruzzaTomassini24} obtained as the quotient of a solvable, non-nilpotent Lie group of complex dimension $4$ with a lattice.  We mention that the Lie group is the $6$th element in the classification of $4$-dimensional complex Lie groups by I. Nakamura in \cite[Section 6]{Nakamura75} and the lattice was constructed in \cite[Section 6]{SferruzzaTomassini24} by following \cite[Example 3.4]{Yamada05}. The group is, in particular, the semidirect-product $\C\ltimes_\varphi N$, where the nilradical $N$ is isomorphic to the $3$-dimensional complex Heisenberg group. Geometrically, this manifold is a holomorphic fibration over a torus with fiber biholomorphic to the Iwasawa manifold. We obtain $4$ classes of deformations, analogous to those described by I. Nakamura in \cite{Nakamura75}. Using the notation and techniques introduced in \cites{RaoZhao15, RaoZhao18_2}, we also compute the structure equations for the first two classes of deformations, establishing the following.

\begin{theorem}[Theorem \ref{teorema 4d}]\label{teo intro 4d}
    Small deformations of the solvmanifold $M_\pi=\Gamma_\pi\backslash G$ defined in \cite[Section 6]{SferruzzaTomassini24} belonging to classes \ref{class14d} and \ref{class24d} do not admit any SKT metric. Moreover, small deformations belonging to class \ref{class14d} do not admit any astheno-K\"ahler metric.
\end{theorem}

The present paper is organized as follows. In Section \ref{Section 2} we introduce the notation adopted throughout the paper and recall some basic results on special Hermitian metrics and on the theory of deformations of complex manifolds. Section \ref{section 3} consists of the proof of the theoretical results on nilmanifolds with structure equations as in \eqref{nilpotent Lie algebra T}, namely Theorem \ref{teorema intro T}, and the construction of new examples of nilmanifolds admitting metrics satisfying condition \eqref{FT metrics} in Example \ref{Condition 1 preservata}. In Section \ref{section 4} we consider the deformations of the Iwasawa manifold and of the holomorphically parallelizable Nakamura manifold computed in \cite{Nakamura75} and  we study the SKT condition on these examples. Lastly, Section \ref{Section Kuranishi} is devoted to the computation of the Kuranishi space of the $4$-dimensional solvmanifold considered in \cite[Section 6]{Sferruzza23}. Furthermore, for classes \ref{class14d} and \ref{class24d} of deformations, we compute the complex structure equations of the deformed manifold, proving Theorem \ref{teo intro 4d}.

\vskip.3truecm
\noindent{\em \underline{Acknowledgments:}} The authors would like to thank the  Erd\H{o}s Center and the R\'enyi Institute, where part of the present paper was written, for the excellent working environment provided. The authors are grateful to Daniele Angella, Hisashi Kasuya and Gueo Grantcharov for their helpful remarks, which contributed to improving the results of the paper. Many thanks also to Nicoletta Tardini, Adela Latorre, Anna Fino, Luigi Vezzoni, and Jonas Stelzig for their valuable comments and interest in the paper.

\section{Preliminaries}\label{Section 2}

\subsection{Preliminaries on special Hermitian metrics}
In this section, we recall some properties of the metrics satisfying condition \eqref{FT metrics} and an obstruction for the existence of metrics such that $\del \delbar \omega^{p} = 0$ on compact complex manifolds. 

Let $(M,J,g)$ be a Hermitian manifold; we denote by $\omega$ the fundamental form of $g$, i.e., $\omega(\cdot , \cdot) = g(J \cdot, \cdot)$. 
As shown in \cite{FinoTomassini11}, it is straightforward to verify that, if $g$ is a metric such that $\del \delbar \omega = 0$, $\del\delbar \omega^{2} = 0$, then $\del \delbar \omega^{p} = 0$ for all $p=1, \dots, n$. Indeed, for $p > 2$, we have
\begin{equation*}
    \del \delbar \omega^{p} = p \, \del (\delbar \omega \wedge \omega^{p-1})  = p \big(\del \delbar \omega \wedge \omega - (p-1) \delbar \omega \wedge \del \omega\big)  \wedge \omega^{p-2},
\end{equation*}
thus $\del \delbar \omega^{p} = 0$.

We recall that since $\omega^{p}$ is a {\em non-degenerate strongly positive} form (see \cite[Corollary 1.10]{HarveyKnapp74}), if $\del \delbar \omega^{p} = 0$, then it is a {\em $p$-pluriclosed structure} on $M$, i.e., it is a $\del \delbar$-closed, transverse $(p,p)$-form. We refer to \cite[Definition 2.1]{Alessandrini17} for the definitions of strongly positive and transverse forms.

An obstruction for the existence of metrics such that $\del \delbar \omega^{p} = 0$ is given by the following result.
\begin{proposition}\cite[Lemma 3.5]{SferruzzaTomassini23}\label{Obstructions to the existence of p pluriclosed}
    Let $(M,J)$ be a compact complex manifold of dimension $n$. Let $\alpha$ be a $(2n -2p - 2)$-form such that 
    \begin{equation*}
        (\del\delbar \alpha)^{n-p,n-p} = \sum_{i} c_{i} \psi^{i} \wedge \overline{\psi^{i}},
    \end{equation*}
    where $\psi^{i}$ are decomposable $(n-p,0)$-forms and $c_{i}$ are constants with the same sign. Then, there are no $p$-pluriclosed structures on $M$.
\end{proposition}

\subsection{Preliminaries on deformation theory}
This section is devoted to recalling some fundamental notions and results in deformation theory and to establishing the notation that will be adopted in the subsequent sections. 

We start by recalling the following definition.
\begin{definition}
    Let $B$ be a domain of $\R^{m}$ (resp. $\C^{m}$) and let $\{M_{t}\}_{t \in B}$ be a set of compact complex manifolds. We say that $\{M_{t}\}_{t \in B}$ is a {\em differentiable (resp. holomorphic) family of compact complex manifolds} if there exists a differentiable (resp. complex) manifold $\mathcal{M}$ and a differentiable (resp. holomorphic), proper map $\pi : \mathcal{M} \to B$ such that 
    \begin{enumerate}
        \item $\pi^{-1}(t) = M_{t}$ as complex manifolds, $\forall t \in B$;
        \item the rank of the Jacobian of $\pi$ is constant and equal to the dimension of $B$, for every $p \in \mathcal{M}$.
    \end{enumerate}
    If $M_{t_{0}} = M$, we say that $M_{t}$ is a deformation of $M$,  $\forall t \in B$.
\end{definition}

The compact complex manifold $(M_{0},J_{0})$ is called the {\em central fiber} of the differentiable (resp. holomorphic) family of compact complex manifolds and it will be denoted by $(M,J)$.

Consider a differentiable family of compact complex manifolds $\{M_{t}\}_{t \in B}$. By a result of C. Ehresmann (see \cite{Ehresmann47}, \cite[Proposition 6.2.2]{Huybrechts05}), $M_{t}$ is diffeomorphic to $M_{\widetilde{t}}$, for all $t, \widetilde{t} \in B$. For simplicity, we assume that $B \doteq \{(t_{1}, \dots, t_{m}) \in \R^{m} \, |  \, |t_{j}| < \varepsilon, j =1, \dots m\}\subseteq \R^{m}$, for $\varepsilon > 0$ and that $t_{0} = 0$. We recall that if the real dimension of $B$ is $1$, the differentiable family $\{M_{t}\}_{t \in I}$ is called a {\em curve of deformations} of $(M,J)$.

\medskip

By \cite[Chapter 4, Section 4.1 (b)]{Kodaira05} and \cite{KodairaMorrow06}, we can parametrize the complex structure on each $M_{t}$, for $t \in B$, with a $(0,1)$-vector form $\varphi(t)$ on $(M,J)$, i.e., $\varphi(t) \in \mathcal{A}^{0,1}(T^{1,0}(M))$. In particular, 
\begin{equation*}
    \varphi(0) = 0, \quad \delbar \varphi(t) - \frac{1}{2}[\varphi(t),\varphi(t)] = 0,
\end{equation*}
where $[\cdot\, , \cdot]$ is a suitable bracket. The second equation is known as the Maurer-Cartan equation for  $\varphi(t)$, which is equivalent to the integrability of the deformed complex structure. If $(z_{1}, \dots, z_{n})$ denotes the local holomorphic coordinates on $(M,J)$, then by \cite[Chapter 4]{KodairaMorrow06}, the local expression of $\varphi(t)$ is given by 
\begin{equation*}
    \varphi(t) = \sum_{\lambda = 1}^{n} \varphi^{\lambda} \otimes \frac{\del}{\del z^{\lambda}},
\end{equation*}
where $\varphi^{\lambda}$ is a global $(0,1)$-form on $(M,J)$. 

Moreover, from the theory of infinitesimal deformations developed by K. Kodaira and D. C. Spencer, it turns out that a differentiable function $f$ defined on the central fiber is holomorphic with respect to the complex structure of $M_{t}$ if and only if 
\begin{equation*}
    \big(\delbar -  \sum_{\lambda=1}^{n} \varphi^{\lambda} \otimes \frac{\del}{\del z^{\lambda}} \big) f  = 0
\end{equation*}
(see \cite[Chapter 4, Proposition 1.2]{KodairaMorrow06}).
\medskip

We now recall the notation and some results established in \cites{RaoZhao15, RaoZhao18_2}. These results provide a direct link between $\mathcal{A}^{p,q}(M_{t})$ and $\mathcal{A}^{p,q}(M)$. Let $\phi$ be a $(0,1)$-vector form, i.e., $\phi = \eta \otimes Z \in \mathcal{A}^{0,1}(T^{1,0}(M))$, where $\eta \in \mathcal{A}^{0,1}(M)$, $Z \in \Gamma(M,T^{1,0}(M))$. We can define the {\em contraction operator}
\begin{equation*}
    \iota_{\phi} : \mathcal{A}^{p,q}(M) \to \mathcal{A}^{p-1, q+1}(M), \quad \sigma \mapsto \iota_{\phi}(\sigma) = \eta \wedge (\iota_{Z} \sigma),
\end{equation*}
where $\iota_{Z}$ is the usual interior product. We will also use the notation $\phi \lrcorner \, \sigma$ to denote the contraction operator. 

Since $\varphi(t) \in \mathcal{A}^{0,1}(T^{1,0}(M))$, we can define the following operators
\begin{equation}\label{exponential of phi}
    e^{\iota_{\varphi(t)}} \doteq \sum_{k=0}^{\infty} \frac{1}{k!} \iota^{k}_{\varphi(t)}, \quad e^{\iota_{\overline{\varphi(t)}}} \doteq \sum_{k=0}^{\infty} \frac{1}{k!} \iota^{k}_{\overline{\varphi(t)}},
\end{equation}
where $\iota^{k}_{\varphi(t)}$ denotes the contraction by $\varphi(t)$ applied $k$-times, and $\overline{\varphi(t)} \in \mathcal{A}^{1,0}(T^{0,1}(M))$ is the conjugate of $\varphi(t)$. Note that the sums are indeed finite since $M$ is finite dimensional. These operators were used in \cite{RaoZhao15} to introduce the map
\begin{equation*}
    e^{\iota_{\varphi(t)}|\iota_{\overline{\varphi(t)}}} : \mathcal{A}^{p,q}(M) \to \mathcal{A}^{p,q}(M_{t}),
\end{equation*}
defined in the following way: if $\sigma \in \mathcal{A}^{p,q}(M)$ is locally written as $\sigma = \sigma_{i_{1} \dots i_{p} j_{1} \dots j_{q}} dz^{i_{1}} \wedge \dots \wedge dz^{i_{p}} \wedge dz^{\overline{j_{1}}} \wedge \dots \wedge dz^{\overline{j_{q}}}$, then 
\begin{equation*}
    e^{\iota_{\varphi(t)}|\iota_{\overline{\varphi}}} (\sigma) \doteq \sigma_{i_{1} \dots i_{p} j_{1} \dots j_{q}}  e^{\iota_{\varphi(t)}}(dz^{i_{1}} \wedge \dots \wedge dz^{i_{p}}) \wedge (e^{\iota_{\overline{\varphi(t)}}} (dz^{\overline{j_{1}}} \wedge \dots \wedge dz^{\overline{j_{q}}}) ).
\end{equation*}
\begin{lemma}\label{Isomorfismo tra lo spazio delle forme}\cite[Lemma 2.9]{RaoZhao18_2}
    The map $e^{\iota_{\varphi(t)}|\iota_{\overline{\varphi(t)}}} : \mathcal{A}^{p,q}(M) \to \mathcal{A}^{p,q}(M_{t})$ is a real linear isomorphism for $t$ arbitrarily small and for any $p,q$.
\end{lemma}

Moreover, by using the following operator, which is called {\em simultaneous contraction},
\begin{equation*}
    \varphi(t) \Finv \, \sigma \doteq \sigma_{i_{1} \dots i_{p} j_{1} \dots j_{q}} \iota_{\varphi(t)} (dz^{1}) \wedge \dots \wedge \iota_{\varphi(t)} ( dz^{i_{p}}) \wedge \iota_{\varphi(t)} (dz^{\overline{j_{1}}})\wedge \dots \wedge \iota_{\varphi(t)}( dz^{\overline{j_{q}}}).
\end{equation*}
we can rewrite $e^{\iota_{\varphi(t)}|\iota_{\overline{\varphi(t)}}}$ as 
\begin{equation*}
    e^{\iota_{\varphi(t)}|\iota_{\overline{\varphi(t)}}} = \big(\text{I} + \varphi(t) + \overline{\varphi}(t) \big) \Finv.
\end{equation*}

By \cite[Proposition 2.7, Proposition 2.13]{RaoZhao18_2}, we can express the action of the operator $\del_{t}$ and $ \delbar_{t}$ on differentiable functions on $M$ with values in $\C$ in the following way: 
\begin{equation*}
    \begin{split}
        & \del_{t} f = e^{\iota_{\varphi}} \Big(\big(I - \varphi \overline{\varphi} \big)^{-1} \lrcorner \, \big(\del - \overline{\varphi} \lrcorner \, \delbar \big) f \Big), \\
        & \delbar_{t} f =  e^{\iota_{\overline{\varphi}}} \Big(\big(I - \overline{\varphi} \varphi  \big)^{-1} \lrcorner \, \big(\delbar - \varphi \lrcorner \, \del \big) f \Big),
    \end{split}
\end{equation*}
where $\varphi \overline{\varphi} \doteq \overline{\varphi} \lrcorner \, \varphi$, $\overline{\varphi} \varphi \doteq \varphi \lrcorner \, \overline{\varphi}$ and $\varphi \doteq \varphi(t)$. Finally, the explicit expressions of $\del_{t}$ and $ \delbar_{t}$ on $(p,q)$-forms on $M_{t}$ written as $e^{\iota_{\varphi}|\iota_{\overline{\varphi}}} \alpha$ for $\alpha \in \mathcal{A}^{p,q}(M)$ is the following:
\begin{equation}\label{del operator on M_t}
    \del_{t} \big(e^{\iota_{\varphi}|\iota_{\overline{\varphi}}} \alpha \big) = e^{\iota_{\varphi}|\iota_{\overline{\varphi}}} \Big(\big(I - \varphi \overline{\varphi} \big)^{-1} \Finv \, \big([\delbar,\iota_{\overline{\varphi}}] + \del \big) (I - \varphi \overline{\varphi} \big) \Finv \, \alpha \Big),
\end{equation}
\begin{equation}\label{delbar operator on M_t}
    \delbar_{t} \big(e^{\iota_{\varphi}|\iota_{\overline{\varphi}}} \alpha \big) = e^{\iota_{\varphi}|\iota_{\overline{\varphi}}} \Big(\big(I - \overline{\varphi} \varphi \big)^{-1} \Finv \, \big([\del,\iota_{\varphi}] + \delbar \big) (I - \overline{\varphi} \varphi \big) \Finv \, \alpha \Big).
\end{equation}

\begin{remark}
    Using the linear isomorphism $e^{\iota_{\varphi(t)}|\iota_{\overline{\varphi(t)}}} : \mathcal{A}^{p,q}(M) \to \mathcal{A}^{p,q}(M_{t})$, we can reformulate the Maurer-Cartan equation for $\varphi(t)$ in a more suitable way. If $\{\psi^{1}, \dots, \psi^{n}\}$ is a global coframe of $(1,0)$-forms on $(M,J)$, then the Maurer-Cartan equation is satisfied if 
    \begin{equation*}
        (d \psi_{t}^{j})^{0,2} = 0, \quad \forall j=1, \dots, n,
    \end{equation*}
    where $\psi_{t}^{j} \doteq e^{\iota_{\varphi(t)}|\iota_{\overline{\varphi(t)}}} (\psi^{j})$, for $j=1, \dots, n$.
\end{remark}

\section{Special Hermitian metrics}\label{section 3}

\subsection{Existence results for SKT and astheno-K\"ahler metrics}

In this section, we study the existence of SKT metrics and metrics satisfying condition \eqref{FT metrics} on a class of nilmanifolds equipped with $2$-step nilpotent complex structures. 

Let $(\Gamma \backslash G,J)$ be a nilmanifold equipped with a left-invariant complex structure $J$ and denote by $\mathfrak{g}$ the Lie algebra of $G$. We recall that the ascending series $\{\mathfrak{g}_{i}^{J}\}$ (compatible with the complex structure $J$) is defined by
\begin{equation*}
    \mathfrak{g}^{J}_{0} \doteq \{0\}, \quad \mathfrak{g}^{J}_{i} \doteq \{X \in \mathfrak{g} \enskip | \enskip [X, \mathfrak{g}]\subseteq \mathfrak{g}^{J}_{i-1}, \quad [J X, \mathfrak{g}] \subseteq \mathfrak{g}^{J}_{i-1} \}.
\end{equation*}
The complex structure $J$ is called {\em nilpotent} (see \cites{CorderoFernandezGrayUgarte97, CorderoFernandezGrayUgarte00}) if there exists a positive integer $t>0$ such that $\mathfrak{g}^{J}_{t} = \mathfrak{g}$. Moreover, an equivalent condition for the nilpotency of $J$ is that there exists a coframe of left-invariant $(1,0)$-forms $\{\varphi^{i}\}_{i=1, \dots n}$  such that 
\begin{equation}\label{J-nilpotent}
    d \varphi^{s} = \sum_{j<k<s} a_{sjk} \varphi^{j k} + \sum_{j,k <s} b_{sj\overline{k}} \varphi^{j \overline{k}}, \quad s=1,\dots n,
\end{equation}
where $a_{sjk}, b_{sj\overline{k}} \in \C$. 

In the following, we focus on nilmanifolds $(M \doteq\Gamma \backslash G, J)$ of complex dimension $n$ equipped with a left-invariant complex structure $J$ admitting a $(1,0)$-coframe $\{\varphi^{1}, \dots, \varphi^{n}\}$ that satisfies the following complex structure equations:
\begin{equation}\label{Nilpotent Lie algebra sezione 3}
    \begin{cases}
        d \varphi^{j} = 0 , \quad  j=1, \dots, n-1, \\[3pt]
        d \varphi^{n} \in \Lambda^{2} \langle \varphi^{1}, \dots, \varphi^{n-1}, \varphi^{\overline{1}}, \dots, \varphi^{\overline{n-1}} \rangle.
    \end{cases}
\end{equation}
We begin by recalling a result of A. Fino, M. Parton and S. Salamon, which characterizes the existence of SKT metrics on nilmanifolds of complex dimension $3$.

\begin{theorem}\cite[Theorem 1.2]{FinoPartonSalamon04} \label{FPS Manifold}
    Let $(M \doteq \Gamma \backslash G)$ be a $6$-dimensional nilmanifold with an invariant complex structure $J$. Then the SKT condition is satisfied by either all invariant Hermitian metrics $g$ on $M$ or by none. Indeed, it is satisfied if and only if $J$ has a basis $(\alpha^{i})$ of $(1,0)$-forms such that 
    \begin{equation*}
        \begin{cases}
            d \alpha^{1} = 0, \\[3pt]
            d \alpha^{2} = 0, \\[3pt]
            d \alpha^{3} = A \alpha^{\overline{1}2} + B \alpha^{\overline{2}2} + C \alpha^{1 \overline{1}} + D \alpha^{1 \overline{2}} + E \alpha^{12},
        \end{cases}
    \end{equation*}
    where $A,B,C,D,E$ are complex numbers such that 
    \begin{equation*}
        |A|^{2} + |D|^{2} + |E|^{2} +  2 \mathfrak{Re}(\overline{B}C) = 0.
    \end{equation*}
\end{theorem}

The following result shows that for nilmanifolds with structure equations as in \eqref{Nilpotent Lie algebra sezione 3}, either every left-invariant Hermitian metric is SKT or none is. The proof can be deduced from the arguments contained in the proof of \cite[Theorem 7.4]{SferruzzaTomassini23}; however, for the sake of completeness, we provide a simple proof here which uses a similar argument.

\begin{lemma}\label{tutte skt}
    Let $(M \doteq\Gamma \backslash G, J)$ be a nilmanifold of complex dimension $n$ equipped with a left-invariant complex structure $J$ satisfying the structure equations \eqref{Nilpotent Lie algebra sezione 3}. Then, every left-invariant Hermitian metric on $M$ is SKT or none is. 
\end{lemma}
\begin{proof}
    Let $\omega$ be the fundamental form of a left-invariant Hermitian metric $g$. We can write
    \begin{equation*}
        \omega = \sum_{j,k=1}^{n} x_{j \overline{k}} \varphi^{j \overline{k}},
    \end{equation*}
    where $x_{j\overline{k}} \in \C$ with $\overline{{x_{k \overline{j}}}} = - x_{j \overline{k}}$. Since $- i x_{n \overline{n}} = \omega(Z_{n}, J \overline{Z}_{n}) > 0$, where $\{Z_{1}, \dots, Z_{n}\}$ is the dual basis to $\{\varphi^{1}, \dots, \varphi^{n}\}$, and $\del\delbar\varphi^{j\overline{k}}=0$ for $(j,k)\neq(n,n)$, it follows that $\del \delbar \omega = x_{n \overline{n}} \del \delbar \varphi^{n \overline{n}}$. 
    Therefore, a left-invariant Hermitian metric $g$ on $M$ is SKT if and only if $\del \delbar \varphi^{n \overline{n}} = 0$. Since $g$ is arbitrary, the thesis follows.
\end{proof} 
In the same hypotheses, we are able to prove the following result.  
\begin{theorem}\label{Tutte FT}
    Let $(M \doteq\Gamma \backslash G, J)$ be a nilmanifold of complex dimension $n$ equipped with a left-invariant complex structure $J$ that satisfies \eqref{Nilpotent Lie algebra sezione 3}. If $(M,J)$ admits a left-invariant SKT metric, then every left-invariant metric satisfies condition \eqref{FT metrics}. 
\end{theorem}
\begin{proof}
    As shown in the proof of Lemma \ref{tutte skt}, the existence of a left-invariant SKT metric implies that $\del \delbar \varphi^{n \overline{n}} = 0 $. If we denote by $\omega$ the fundamental form of a left-invariant Hermitian metric $g$, then
    \begin{equation*}
        \frac{1}{2}\omega^{2} = \sum_{j < k, r < s} x_{j \overline{r}} x_{k \overline{s}} \varphi^{j \overline{r} k \overline{s}} = \sum_{\substack{j < k,  r < s,\\ (k,\overline{s})\neq(n,\overline{n})}} x_{j \overline{r}} x_{k \overline{s}} \varphi^{j \overline{r} k \overline{s}} + \varphi^{n \overline{n}}\wedge \Big(\sum_{j=1, r = 1}^{n -1} a_{j \overline{r}} \varphi^{j \overline{r}} \Big),
    \end{equation*}
    where $a_{j \overline{r}}\in\C$. Since $\del\delbar\varphi^{j\overline{k}}=0$ for all $(j,k)$, then $\del \delbar \omega^{2} = 0$ and so, as observed in the preliminaries, $\del \delbar \omega^{l} = 0$, for all $ l=1, \dots, n$. 
\end{proof}

\subsection{Deformations of metrics satisfying condition $\del \delbar \omega = 0, \del \delbar \omega^{2} = 0$}\label{Subsection deformation}

Let $\{M_{t}\}_{t \in I}$, where $I \doteq (- \varepsilon, \varepsilon)$ for $\varepsilon > 0$, be a curve of deformations of a compact complex manifold $(M,J)$ admitting a Hermitian metric $g$ such that $\del \delbar \omega^{k} = 0$ for a certain $k \geq 1$, where $\omega$ denotes the fundamental form of $g$. Suppose that $\{M_{t}\}_{t \in I}$ is parametrized by $\varphi(t) \in \mathcal{A}^{0,1}(T^{1,0}M)$. By Lemma \ref{Isomorfismo tra lo spazio delle forme}, a smooth family of Hermitian metrics $\{\omega_{t}\}_{t \in I}$ along $\{M_{t}\}_{t \in I}$ such that $\omega_{0} = \omega$ can be locally written as $\omega_{t} = e^{\iota_{\varphi(t)}|\iota_{\overline{\varphi(t)}}} \omega(t)$, where $\omega(t) = \omega_{ij}(t) dz^{i} \wedge d\overline{z}^{j} \in \mathcal{A}^{1,1}(M)$. Moreover, 
\begin{equation}\label{potenza di omega al tempo t}
    \omega^{k}_{t} = e^{\iota_{\varphi(t)}|\iota_{\overline{\varphi(t)}}} \big( \omega^{k}(t) \big),
\end{equation}
where, locally 
\begin{equation*}
    \omega^{k}(t) = \sigma(t)_{i_{1}j_{1} \dots i_{k}j_{k}}dz^{i_{1}} \wedge d\overline{z}^{j_{1}} \wedge \dots \wedge dz^{i_{k}} \wedge d\overline{z}^{j_{k}},
\end{equation*}
and $\sigma(t)_{i_{1}j_{1} \dots i_{k}j_{k}} \doteq \omega_{i_{1}j_{1}}(t) \cdot \dots \cdot \omega_{i_{k}j_{k}}(t)$. For later purposes, we set 
\begin{equation*}
    (\omega^{k}(t))' \doteq \frac{\del}{\del t} (\sigma(t)_{i_{1}j_{1} \dots i_{k}j_{k}}) dz^{i_{1}} \wedge d\overline{z}^{j_{1}} \wedge \dots \wedge dz^{i_{k}} \wedge d\overline{z}^{j_{k}} \; \in \; \mathcal{A}^{k,k}(M).
\end{equation*}

We recall that in \cites{PiovaniSferruzza21, Sferruzza23} the authors developed techniques to provide necessary conditions for the existence of a smooth family of SKT and astheno-K\"ahler Hermitian metrics along $\{M_{t}\}_{t \in I}$. Moreover, in \cite{Ciulica25}, the author provides necessary conditions for the existence of a smooth family of Hermitian metrics $\{\omega_{t}\}_{t \in I}$ along $\{M_{t}\}_{t \in I}$ such that $\del_{t} \delbar_{t} \omega^{k}_{t} = 0$ and $\omega_{0} = \omega$. Such conditions are recalled in the following.

\begin{theorem}\label{Obstructions to the existence of curves}\cite[Theorem 3.21]{Ciulica25}
    Let $\{M_{t}\}_{t \in I}$ be a curve of deformations of a compact complex manifold $(M,J)$ admitting a Hermitian metric $g$ such that its fundamental form satisfies $\del \delbar \omega^{k} = 0$ for a certain $k \geq 1$. Suppose that the deformation is parametrized by $\varphi(t) \in \mathcal{A}^{0,1}(T^{1,0}M)$ and that $\{\omega_{t}\}_{t \in I}$ is a smooth family of Hermitian metric written as in \eqref{potenza di omega al tempo t} such that $\del_{t} \delbar_{t} \omega^{k}_{t} = 0$, $\forall t \in I$. Then 
    \begin{equation*}
        2 i \mathfrak{Im} (\del \circ \iota_{\varphi(0)'} \circ \del) (\omega^{k}) = \del \delbar(\omega^{k}(0))'.
    \end{equation*}
\end{theorem}

A direct consequence of Theorem \ref{Obstructions to the existence of curves} is the following.

\begin{corollary}\cite[Corollary 3.22]{Ciulica25}
    Let $\{M_{t}\}_{t \in I}$ be a curve of deformations of a compact complex manifold $(M,J)$ admitting a Hermitian metric $g$ such that its fundamental form satisfies $\del \delbar \omega^{k} = 0$. Suppose that the deformation is parametrized by $\varphi(t) \in \mathcal{A}^{0,1}(T^{1,0}M)$ and that $\{\omega_{t}\}_{t \in I}$ is a smooth family of Hermitian metric written as in \eqref{potenza di omega al tempo t} such that $\del_{t} \delbar_{t} \omega^{k}_{t} = 0$, $\forall t \in I$. Then 
    \begin{equation*}
        [\mathfrak{Im}(\del \circ \iota_{\varphi(0)'} \circ \del) (\omega^{k})]_{H_{BC}^{k+1,k+1}(M)} = 0.
    \end{equation*}
\end{corollary}

These results also apply to Hermitian metrics that satisfy the condition \eqref{FT metrics}.

\begin{corollary}\label{Vanishing of Bott Chern cohomology class}\cite[Corollary 3.24]{Ciulica25}
    Let $\{M_{t}\}_{t \in I}$ be a curve of deformations of a compact complex manifold $(M,J)$ admitting a Hermitian metric $g$ such that its fundamental form satisfies $\del \delbar \omega = 0$, $\del \delbar \omega^{2} = 0$. Suppose that the deformation is parametrized by $\varphi(t) \in \mathcal{A}^{0,1}(T^{1,0}M)$ and that $\{\omega_{t}\}_{t \in I}$ is a smooth family of Hermitian metric written as in \eqref{potenza di omega al tempo t} such that $\del_{t} \delbar_{t} \omega_{t} = 0$, $\del_{t} \delbar_{t} \omega^{2}_{t} = 0$, $\forall t \in I$. Then 
    \begin{equation*}
        [\mathfrak{Im}(\del \circ \iota_{\varphi(0)'} \circ \del) (\omega^{k})]_{H_{BC}^{k+1,k+1}(M)} = 0, \quad  \forall k \geq 1.
    \end{equation*}
\end{corollary}

As an application of the above results, we provide examples of deformations that preserve condition \eqref{FT metrics}.

\begin{example}\label{Condition 1 preservata}
The $2$-step nilpotent Lie algebra $\mathfrak{g}$ studied in \cite[Theorem 2.7]{FinoTomassini11}, can be endowed with an integrable complex structure $J$ such that the fundamental form of the diagonal Hermitian metric satisfies the condition \eqref{FT metrics}. With this choice of $J$, we have that $(\mathfrak{g}^{1,0})^{\ast}$ is spanned by $\{\varphi^{1}, \dots, \varphi^{4}\}$, with structure equations
\begin{equation}\label{astheno + SKT}
    \begin{cases}
            d \varphi^{j}=0, \quad  j= 1, 2, 3, \\[3pt]
            d \varphi^{4} = a_{2} \, \varphi^{13} + a_{3} \, \varphi^{1 \overline{1}} + a_{5} \, \varphi^{1 \overline{3}} + a_{10} \, \varphi^{3 \overline{1}} + a_{12} \, \varphi^{3 \overline{3}},
        \end{cases}
\end{equation}
where $a_{j} \in \Q[i]$ for $j=2,3,5,10,12$ satisfy 
\begin{equation}\label{Condition on the coefficient for astheno + SKT}
    |a_{2}|^{2} + |a_{5}|^{2} + |a_{10}|^{2} = 2 \mathfrak{Re} (a_{3}\overline{a_{12}}).
\end{equation}
Let $G$ be the simply connected nilpotent Lie group associated to the Lie algebra $\mathfrak{g}$. By Malcev's Theorem (see \cite{Malcev62}), the existence of a lattice $\Gamma$ in $G$ is ensured by the fact that $a_{j} \in \Q[i]$ for $j=2,3,5,10,12$. Moreover, condition \eqref{Condition on the coefficient for astheno + SKT} must hold if we want the diagonal Hermitian metric to be both SKT and astheno-K\"ahler. The nilmanifold equipped with the left-invariant complex structure $J$ defined by \eqref{astheno + SKT} will be denoted by $(M \doteq \Gamma \backslash G, J)$.

Let us consider the deformation parametrized by the $(0,1)$-vector form studied in \cite[Example 1]{PiovaniSferruzza21}, which is  
\begin{equation*}
    \Psi(r,s) \doteq r \varphi^{\overline{1}} \otimes Z_{1} + s \varphi^{\overline{3}} \otimes Z_{3}, 
\end{equation*}
where $r, s \in \C^{2}$ satisfy $|r|, |s| < 1$ and $Z_{j}$ are dual to $\varphi^{j}$. By Lemma \ref{Isomorfismo tra lo spazio delle forme}, a coframe for $(T^{1,0}M_{t})^{\ast}$ is given by 

\begin{equation*}
    \begin{cases}
        \varphi^{1}_{r,s} = \varphi^{1} + r \varphi^{\overline{1}}, \\[3pt]
        \varphi^{2}_{r,s} = \varphi^{2}, \\[3pt]
        \varphi^{3}_{r,s} = \varphi^{3} + s \varphi^{\overline{3}},  \\[3pt]
        \varphi^{4}_{r,s} = \varphi^{4}.
    \end{cases}
\end{equation*}
As shown in \cite[Example 1]{PiovaniSferruzza21}, the structure equations of $\{\varphi^{1}_{r,s}, \dots, \varphi^{4}_{r,s}\}$ are given by 
\begin{equation*}
    \begin{cases}
        d \varphi^{j}_{r,s} = 0, \quad j = 1,2,3, \\[3pt]
        d\varphi^{4}_{r,s} =\frac{(a_{2} + \overline{r} a_{10} - \overline{s}a_{5} )}{(1-|r|^{2})(1-|s|^{2})} \varphi^{13}_{r,s} + \frac{a_{3}}{(1-|r|^{2})} \varphi^{1 \overline{1}}_{r,s} + \frac{(a_{5} - sa_{2} - \overline{r}s a_{10})}{(1-|r|^{2})(1-|s|^{2})} \varphi^{1 \overline{3}}_{r,s} \\[3pt]
        \quad \quad \quad \;+ \frac{(a_{10} + r a_{2} - r \overline{s}a_{5})}{(1-|r|^{2})(1-|s|^{2})} \varphi^{3\overline{1}}_{r,s} + \frac{a_{12}}{(1-|s|^{2})} \varphi^{3 \overline{3}}_{r,s} + \frac{(- r a_{5}  + s a_{10} + rs a_{2})}{(1-|r|^{2})(1-|s|^{2})} \varphi^{\overline{13}}_{r,s},
    \end{cases}
\end{equation*}
and the integrability of the complex structure is guaranteed if 
\begin{equation}\label{integrability conditions}
    - r a_{5}  + s a_{10} + rs a_{2} = 0.
\end{equation}
Assume $a_{5}, a_{10} \neq 0$. For $\delta, \delta' > 0$ sufficiently small, the set of solutions of \eqref{integrability conditions} is given by 
\begin{equation*}
    A = \{(r,s) \in \C^{2} \; | \; r = \frac{sa_{10}}{a_{5} - s a_{2}}, |r| < \delta, |s| < \delta'\}.
\end{equation*}
Consider the curve 
\begin{equation*}
\gamma : (- \varepsilon, \varepsilon) \to A , \quad t \mapsto \left(\frac{t u a_{10}}{a_{5} - t u a_{2}}, tu\right),
\end{equation*}
where $u \in \C$. Thus, the curve of deformations 
\begin{equation}\label{curva di deformazioni primo esempio}
    \Psi(t) = \frac{t u a_{10}}{a_{5} - t u a_{2}} \varphi^{\overline{1}} \otimes Z_{1} + tu \varphi^{\overline{3}} \otimes Z_{3}
\end{equation}
is such that $\Psi(0)' = \frac{u a_{10}}{a_{5}} \varphi^{\overline{1}} \otimes Z_{1} + u \varphi^{\overline{3}} \otimes Z_{3}$. Denote by $\omega$ the fundamental form of the diagonal Hermitian metric. Since
\begin{equation*}
    \del \circ \iota_{\Psi(0)'} \circ \del (\omega) = i u a_{2} \frac{|a_{10}|^{2} - |a_{5}|^{2}}{a_{5}} \eta^{13 \overline{13}}, \quad \del \circ \iota_{\Psi(0)'} \circ \del (\omega^{2}) = u a_{2} \frac{|a_{10}|^{2} - |a_{5}|^{2}}{a_{5}} \eta^{123 \overline{123}}
\end{equation*}
and both $\eta^{13 \overline{13}}$ and $\eta^{123 \overline{123}}$ are harmonic with respect to the Bott-Chern Laplacian, we obtain, by Corollary \ref{Vanishing of Bott Chern cohomology class}, that if 
\begin{equation*}
    u a_{2} (|a_{10}|^{2} - |a_{5}|^{2}) \neq 0,
\end{equation*}
there are no curves of metrics $\{\omega_{t}\}$ such that \eqref{FT metrics} holds and $\omega_{0} = \omega$. 

In the following, we want to provide conditions on the coefficients $a_{2}, a_{5}, a_{10}$ in order to preserve the condition \eqref{FT metrics} along the deformation \eqref{curva di deformazioni primo esempio}. We consider the case $a_{2} = 0 $.

By direct computation, the structure equations of $\{\varphi^{1}_{r,s}, \dots, \varphi^{4}_{r,s}\}$ are:
\begin{equation*}
    \begin{cases}
        d \varphi^{j}_{t} = 0, \quad j = 1,2,3, \\[3pt]
        d\varphi^{4}_{t} = A_{t} \varphi^{13}_{t} + B_{t} \varphi^{1 \overline{1}}_{t} + C_{t} \varphi^{1 \overline{3}}_{t} + D_{t} \varphi^{3\overline{1}}_{t} + E_{t} \varphi^{3 \overline{3}}_{t},
    \end{cases}
\end{equation*}
where 
\begin{align*}
    & A_{t} \doteq \frac{t\overline{u} a_{5} (|a_{10}|^{2} - |a_{5}|^{2})}{(|a_{5}|^{2} - |tu|^{2} |a_{10}|^{2}) (1 - |tu|^{2})}, \quad && B_{t} \doteq \frac{a_{3} |a_{5}|^{2}}{|a_{5}|^{2} - |tu|^{2} |a_{10}|^{2}}, \quad C_{t} \doteq \frac{a_{5}(|a_{5}|^{2} - |tu|^{2} |a_{10}|^{2})}{(|a_{5}|^{2} - |tu|^{2} |a_{10}|^{2}) (1 - |tu|^{2})}, \\
    & D_{t} \doteq \frac{a_{10}|a_{5}|^{2}}{(|a_{5}|^{2} - |tu|^{2} |a_{10}|^{2})}, \quad && E_{t} \doteq \frac{a_{12}}{1 - |tu|^{2}}.
\end{align*}
It is straightforward to check that 
\begin{equation*}
    \del_{t} \delbar_{t} \varphi^{4 \overline{4}}_{t} = 2 S^{2} |tu|^{2} |a_{5}|^{2} (|a_{10}|^{2} - |a_{5}|^{2})^{2} \varphi_{t}^{1 \overline{1} 3 \overline{3}},
\end{equation*}
where $S = \frac{1}{(|a_{5}|^{2} - |tu|^{2} |a_{10}|^{2}) (1 - |tu|^{2})}$. Thus, if $|a_{10}|^{2} = |a_{5}|^{2}$, we conclude that $\del_{t} \delbar_{t} \varphi^{4 \overline{4}}_{t} = 0$ and so, by Theorem \ref{Tutte FT}, every left-invariant metric along small deformation of $(M,J)$ satisfies condition \eqref{FT metrics}.
\end{example}

\section{Deformations of special Hermitian metrics}\label{section 4}

\subsection{Deformations of the Iwasawa manifold}\label{subsection Iwasawa}

We now study the SKT condition on deformations of the Iwasawa manifold. We recall that the Iwasawa manifold is defined as the compact quotient ${\mathbb{I}}_3= \Gamma\backslash\mathbb{H}(3;\C)$, where $\mathbb{H}(3;\C)$ is the Heisenberg group over $\C$ defined by
\[
\mathbb{H}(3;\C) := \left\{
\left(
\begin{array}{ccc}
 1 & z_1 & z_3 \\
 0 &  1  & z_2 \\
 0 &  0  &  1
\end{array}
\right) \in \mathrm{GL}(3;\C)\;:\; z_1,\,z_2,\,z_3 \in\C \right\},
\;
\]
with the product induced by matrix multiplication and $ \Gamma = \mathbb{H}(3;\Z[i])\doteq\mathbb{H}(3;\C)\cap\mathrm{GL}(3;\Z[i])$ acts on the Heisenberg group by left multiplication. We have that 
\[
\begin{cases}
    \varphi^1=dz_1, \\[3pt]
    \varphi^2=dz_2, \\[3pt]
    \varphi^3=dz_3-z_1dz_2,
\end{cases}
\]
is a left-invariant coframe for the space of $(1,0)$-forms on $\mathbb{H}(3;\C)$ and the structure equations with respect to this coframe are 
\[
\begin{cases}
    d\varphi^1=0, \\[3pt]
    d\varphi^2=0, \\[3pt]
    d\varphi^3=-\varphi^1\wedge\varphi^2.
\end{cases}
\]
Since the forms $\varphi^1$, $\varphi^2$ and $\varphi^3$ are $\mathbb{H}(3;\C)$-left-invariant, they define a coframe also for the Iwasawa manifold. In \cite{Nakamura75}, I. Nakamura computed the small deformations of
the Iwasawa manifold: by \cite[page 95]{Nakamura75}, a local system of complex coordinates for the complex structure
at $t=\,\left(t_{11},\,t_{12},\,t_{21},\,t_{22},\,t_{31},\,t_{32}\right)\in\C^6$ is given by
\[
\begin{cases}
 \zeta^t_1 = z_1\,+\,\sum_{j=1}^{2}t_{1j}\,\overline{z}_j, \\[3pt]
 \zeta^t_2 = z_2\,+\,\sum_{j=1}^{2}t_{2j}\,\overline{z}_j, \\[3pt]
 \zeta^t_3 = z_3\,+\,\sum_{j=1}^{2}\big(t_{3j}+t_{2j}\,z_1\big)\overline{z}_j\,+
 \,A(\overline{\mathbf{z}})\,-\,D(t)\,\overline{z}_3,
\end{cases}
\]
where
\begin{eqnarray*}
A(\overline{\mathbf{z}})&\doteq& \frac{1}{2}\left(t_{11}\,t_{21}\,\overline{z}_1\,\overline{z}_1+
2\,t_{11}\,t_{22}\,\overline{z}_1\,\overline{z}_2+t_{12}\,t_{22}\,\overline{z}_2\,\overline{z}_2\right),\\[3pt]
D(t)&\doteq& t_{11}\,t_{22}-t_{12}\,t_{21} .
\end{eqnarray*}
It can be shown that $\C^{3}$ is the universal cover of $M_{t}$ and that $M_{t} = \Gamma_{t} \backslash \C^{3}$ for $\Gamma_{t}$ defined by 
\begin{equation*}
    \begin{cases}
        \zeta^{'}_{1} = \zeta_{1}^{t} + \widetilde{\omega}_{1}, \\[3pt]
        \zeta^{'}_{2} = \zeta_{2}^{t} + \widetilde{\omega}_{2} , \\[3pt]
        \zeta^{'}_{3} = \zeta_{3}^{t} + \widetilde{\omega}_{3} + \omega_{1} \zeta^{t}_{2} + \Big( \sum_{\lambda = 1}^{2} t_{2 \lambda} \overline{\omega}_{\lambda}\Big) (\zeta^{t}_{1} + \omega_{1}) + A(\overline{\omega}) - D(t) \overline{\omega}_{3}, 
    \end{cases}
\end{equation*}
where $\widetilde{\omega}_{j} \doteq \omega_{j} + t_{j1} \overline{\omega}_{1} + t_{j 2} \overline{\omega}_{2}$ and $(\omega_{1}, \omega_{2}, \omega_3) \in \Gamma$. Note that the following $(1, 0)$-forms on $\C^{3}$ 
\begin{equation*}
    \begin{cases}
        \varphi_{t}^{1} \doteq d \zeta^{t}_{1}, \\[3pt]
        \varphi_{t}^{2} \doteq d \zeta^{t}_{2}, \\[3pt]
        \varphi_{t}^{3} \doteq d \zeta^{t}_{3} - z_{1} d \zeta^{t}_{2} - (t_{21} \overline{z}_{1} + t_{22} \overline{z}_{2}) d\zeta^{t}_{1},
    \end{cases}
\end{equation*}
are invariant under the action of $\Gamma_{t}$ and their complex structure equations are given by
\begin{equation}
    \begin{cases}\label{structureIwasawa}
        d\varphi_t^1=0,\\[3pt]
        d\varphi_t^2=0,\\[3pt]
        d\varphi_t^3=\sigma_{12}\varphi_t^{12}+\sigma_{1\overline{1}}\varphi_t^{1\overline{1}}+\sigma_{1\overline{2}}\varphi_t^{1\overline{2}}+\sigma_{2\overline{1}}\varphi_t^{2\overline{1}}+\sigma_{2\overline{2}}\varphi_t^{2\overline{2}}.
    \end{cases}
\end{equation}
Similar computations to those in the Appendix show that $\sigma_{11},\sigma_{1\overline{1}},\sigma_{1\overline{2}},\sigma_{2\overline{1}},\sigma_{2\overline{2}}$ can be written as follows:
\begin{equation*}
    \begin{cases}
        \sigma_{12}=\alpha\gamma\,\big(\,|D(t)|^2-1\big), \\[3pt]
        \sigma_{1\overline{1}}=\alpha\gamma\,\big(\,t_{21}+\overline{t}_{21}D(t)\big), \\[3pt]
        \sigma_{1\overline{2}}=\alpha\gamma\,\big(\,t_{22}-\overline{t}_{11}D(t)\big), \\[3pt]
        \sigma_{2\overline{1}}=\alpha\gamma\,\big(\,\overline{t}_{22}D(t)-t_{11}\big), \\[3pt]
        \sigma_{2\overline{2}}=\alpha\gamma\,\big(\,-t_{12}-\overline{t}_{12}D(t) \big),
    \end{cases}
\end{equation*}
where the values of $\alpha$ and $\gamma$ are as in \cite[page 416]{AngellaTomassini11} and a straightforward computation shows that  
\begin{equation*}
    \alpha \gamma = \frac{1}{1 + |D(t)|^{2} - (|t_{11}|^{2} + |t_{22}|^{2} + t_{12} \overline{t_{21}} + \overline{t_{12}} t_{21})}.
\end{equation*}
By Theorem \ref{FPS Manifold}, $M_{t}$ admits an SKT metric if and only if 
\[
|\sigma_{2\overline{1}}|^2+|\sigma_{1\overline{2}}|^2+|\sigma_{12}|^2-2\mathfrak{Re}\,(\,\overline{\sigma}_{2\overline{2}}\,\sigma_{1\overline{1}})=0. 
\] 
Since $\alpha\gamma\neq0$ for sufficiently small deformations, if we define $E(t)\doteq|t_{11}|^2+|t_{22}|^2+2\mathfrak{Re}\,(\,\overline{t}_{12}t_{21})$, we can rewrite this condition as
\begin{equation}\label{sktIwasawa}
    P(t)=|D(t)|^4+|D(t)|^2 \big(E(t)-6 \big)+E(t)+1=0\;.
\end{equation}
Since $\lim_{t\to0}P(t)=1$, condition \eqref{sktIwasawa} cannot be satisfied for sufficiently small $t$, namely sufficiently small deformations of the canonical complex structure on the Iwasawa manifold cannot admit any SKT metric.

Following \cite[Section 2]{FinoTomassini09}, we focus on deformations for which we have $\sigma_{1\overline{2}}=\sigma_{2\overline{1}}=0$, that is $(t_{22}-\overline{t}_{11}D(t))=(\overline{t}_{22}D(t)-t_{11})=0$. If we suppose $t_{11}$, $t_{22}\neq0$, we have $|t_{11}|=|t_{22}|$ and so $|D(t)|=1$. Thus, condition \eqref{sktIwasawa} becomes $E(t)=2$, that is
    \[
        \mathfrak{Re}\,(\,\overline{t}_{12}t_{21})=1-|t_{22}|^2,
    \]
but, for this choice of $t_{11},t_{22}, t_{12}, t_{21}$, the term $\alpha\gamma$ is not defined.

Suppose that $t_{11}=t_{22}=0$. In this case, $E(t)=2\mathfrak{Re}\,(\,\overline{t}_{12}t_{21})$ and $D(t)=-t_{12}t_{21}$, hence condition \eqref{sktIwasawa} implies
    \[
        2\mathfrak{Re}\,(\,\overline{t}_{12}t_{21})=\frac{-|t_{12}|^4|t_{21}|^4+6|t_{12}|^2|t_{21}|^2-1}{1+|t_{12}|^2|t_{21}|^2}.
    \]
We observe that, if we impose $t_{12}=-t_{21}$ and we denote by $T\doteq|t_{12}|^2=|t_{21}|^2$, condition \eqref{sktIwasawa} is satisfied if and only if 
    \[
        -T^4+2T^3+6T^2+2T-1=0.
    \] 
This equation has real solutions $T=-1, 2-\sqrt{3}, 2+\sqrt{3}$. Since $T$ is non negative, the only solutions we can consider are $T=2-\sqrt{3}, 2+\sqrt{3}$. 

Finally, we note that for $T = 1$ the complex structure is not defined. Indeed, 
\[
\begin{cases}
 \zeta^t_1 = z_1\, - t_{21} \,\overline{z}_2, \\[3pt]
 \zeta^t_2 = z_2\,+\, t_{21}\,\overline{z}_1, \\[3pt]
 \zeta^t_3 = z_3\,+\, t_{21}\,z_1 \overline{z}_1\,+\,t_{21}^{2}\,\overline{z}_3,
\end{cases}
\]
and so 
\[
\begin{cases}
    z_{1} = \frac{1}{1 + |t_{21}|^{2}} (\zeta_{1}^{t} + t_{21} \overline{\zeta^{t}_{2}}), \\[3pt]
    z_{2} = \frac{1}{1 + |t_{21}|^{2}} (\zeta_{2}^{t} + t_{21} \overline{\zeta^{t}_{2}}), \\[3pt]
    z_{3} = \frac{1}{1 - |t_{21}|^{4}} \Big(\zeta_{3}^{t} + t_{21}^{2} \overline{\zeta^{t}_{3}} - \frac{1}{1 + |t_{21}|^{2}} \big(|t_{21}|^{2} \zeta_{1}^{t} \zeta_{2}^{t} + t_{21} \zeta_{1}^{t} \overline{\zeta_{1}^{t}}  + |t_{21}|^{2}t_{21} \zeta_{2}^{t}\overline{\zeta^{t}_{2}} + t_{21}^{2} \overline{\zeta^{t}_{1}} \overline{\zeta^{t}_{2}} \big)  \Big). \\[3pt]
\end{cases}
\]
Hence, the complex structure is not well-defined for $T = |t_{21}|^{2} = 1$. In summary, we have established the following.
\begin{theorem}\label{no SKT intorno a Iwasawa}
    Sufficiently small deformations of the Iwasawa manifold provided in \cite[p. 95]{Nakamura75} cannot admit SKT metrics. 
\end{theorem}

\begin{theorem}\label{curva balanced SKT}
    The deformation of the Iwasawa manifold defined as in \cite[p. 95]{Nakamura75} with $t_{11}=t_{22}=0$ and $t_{12}=-t_{21}$ connects the standard complex structure on the Iwasawa manifold to a circle of complex structures that admits SKT metrics, defined by $|t_{21}|^{2} = T = 2-\sqrt{3}$. Moreover, for $|t_{21}|^{2} = T = 2+\sqrt{3}$, we have another circle of complex structures that admit SKT metrics. 
\end{theorem}

\begin{remark}
    We remark that, for the choice of coefficients $t_{11},t_{22},t_{12},t_{21}$ as in Theorem \ref{curva balanced SKT} there are no balanced metrics. Indeed, the complex structure equations are 
    \begin{equation}
        \begin{cases}
            d\varphi_t^1=0,\\[3pt]
            d\varphi_t^2=0,\\[3pt]
            d\varphi_t^3= \frac{|t_{21}|^{2} - 1}{1 + |t_{21}|^{2}}  \varphi_t^{12} + \frac{t_{21}}{1 + |t_{21}|^{2}}  \varphi_t^{1\overline{1}} + \frac{t_{21}}{1 + |t_{21}|^{2} }\varphi_t^{2\overline{2}},
        \end{cases}
    \end{equation}
    and so  
    \begin{equation*}
        (\overline{t_{21}} \, d \varphi_{t}^{3})^{1,1} = \frac{|t_{21}|^{2}}{1 + |t_{21}|^{2}} \big(\varphi_{t}^{1 \overline{1}} + \varphi_{t}^{2 \overline{2}} \big).
    \end{equation*}
    Thus, by \cite[Proposition 3.4]{HindMedoriTomassini23} there are no balanced metrics on $M_{t}$. In agreement with \cite[Theorem 1.1]{FinoVezzoni2016} for $|t_{21}|^{2} = 2-\sqrt{3},  2+\sqrt{3}$ there are no balanced metrics. 
\end{remark}

We recall that in \cite[Proposition 4.1]{Xiao15} it is shown that strongly Gauduchon metrics are preserved along small deformations of the complex structure. For the choice of coefficients $t_{11}, t_{22}, t_{12}$ and $t_{21}$ as in Theorem \ref{curva balanced SKT}, we can explicitly construct the strongly Gauduchon metrics along $M_{t}$. Indeed, let $\omega_{t}$ be the fundamental form of the diagonal Hermitian metric, i.e., $\omega_{t} \doteq \frac{i}{2} (\varphi^{1 \overline{1}}_{t} + \varphi^{2 \overline{2}}_{t} + \varphi^{3 \overline{3}}_{t})$. Then 
\begin{equation*}
    d \omega_{t}^{2} =  \frac{\overline{t_{21}}}{1 + |t_{21}|^{2}} \varphi_{t}^{123 \overline{12}} + \frac{t_{21}}{1 + |t_{21}|^{2}} \varphi_{t}^{12 \overline{123}}.
\end{equation*}
Moreover, 
\begin{equation*}
        d \Big(\frac{\overline{t_{21}}}{|t_{21}|^{2} - 1} \varphi_{t}^{123 \overline{3}} \Big) = - \frac{\overline{t_{21}}}{1 + |t_{21}|^{2}}  \varphi_{t}^{123 \overline{12}},
\end{equation*}
so $\omega_{t}$ is a strongly Gauduchon metric.

\subsection{Deformations of the holomorphically parallelizable Nakamura manifold}\label{subsection Nakamura}

In the following, we prove that the classes of deformations $(1),(3)$ and $(4)$ (see \cite[Section 3]{Nakamura75}) of the holomorphically parallelizable Nakamura manifold \cite[III-(3b)]{Nakamura75} do not admit any SKT metric \footnote{Part of these computations have been performed by M. G. Franzini in her Master Thesis \cite{Franzini11}, Deformazioni di varietà bilanciate e loro proprietà coomologiche, advisor Prof. A. Tomassini, Università di Parma, 2011; we thank her.}. We recall that the Nakamura manifold of type III-(3b) is defined as $M \doteq \Gamma \backslash \C \ltimes_{\phi} \C^{2}$ where $\phi(z) \doteq \text{Diag}(e^{z}, e^{-z})$ and $\Gamma$ is an appropriate lattice in $\C \ltimes_{\phi} \C^{2}$. In this case, we have that the forms
\begin{equation*}
    \begin{cases}
        \varphi^1=dz_1,\\[3pt]
        \varphi^2=e^{z_{1}}dz_2,\\[3pt]
        \varphi^3=e^{-z_{1}}dz_3, 
    \end{cases}
\end{equation*}
are invariant with respect to the action of $\Gamma$. Hence, they provide a coframe of $(1,0)$-forms for the Nakamura manifold. The dual frame is given by 
\begin{equation*}
    \theta_{1} \doteq \del_{z_{1}}, \quad \theta_{2} \doteq e^{-z_{1}} \theta_{z_{2}}, \quad \theta_{3} \doteq e^{z_{1}} \del_{z_{3}}
\end{equation*}
and the complex structure equations are the following:
\[
\begin{cases}
    d\varphi^1=0,\\[3pt]
    d\varphi^2= \varphi^{1} \wedge \varphi^{2},\\[3pt]
    d\varphi^3= - \varphi^{1} \wedge \varphi^{3}.
\end{cases}
\]
In \cite[Pag. 96]{Nakamura75} it is shown that the $(0,1)$-forms $\widetilde{\varphi}^{1} \doteq\varphi^{\overline{1}}, \widetilde{\varphi}^{2} \doteq e^{z_{1} - \overline{z_{1}}} \varphi^{\overline{2}}, \widetilde{\varphi}^{3} \doteq e^{ \overline{z_{1}} - z_{1}} \varphi^{\overline{3}}$ generate $H_{\delbar}^{0,1}(M)$. Moreover, in \cite[Lemma 3.1]{Nakamura75} it is proved that the deformation defined by
\begin{equation}
    \psi(t) \doteq \sum_{i=1}^{3} \sum_{\lambda=1}^{3} t_{i \lambda} \widetilde{\varphi}^{\lambda} \otimes \theta_{i},
\end{equation}
where $t=t_{j \lambda} \in \C$, satisfies the Maurer-Cartan equation. As shown in \cite[Section 3]{Nakamura75}, the Kuranishi families of the holomorphically parallelizable Nakamura manifold near the origin are parameterized by:
\begin{enumerate}
    \setlength{\itemsep}{.4em}
    \item\label{classe 1 Nakamura} $t_{11} \neq 0$, $t_{12} = t_{13} = t_{23} = t_{32} = 0$;
    \item\label{classe 2 Nakamura} $t_{11} = t_{12} = t_{13} = 0$;
    \item\label{classe 3 Nakamura} $t_{12} \neq 0$, $t_{11} = t_{13} = t_{21} = t_{23} = t_{31} = 0$;
    \item\label{classe 4 Nakamura} $t_{13} \neq 0$, $t_{11} = t_{12} = t_{21} = t_{31} = t_{32} = 0$.
\end{enumerate}

In the following, we recall the fundamental facts about the Kuranishi families of the holomorphically parallelizable Nakamura manifold proved in \cite[Section 3]{Nakamura75} and we prove that the classes \eqref{classe 1 Nakamura}, \eqref{classe 3 Nakamura} and \eqref{classe 4 Nakamura} do not admit any SKT metric.

We recall that for the case \eqref{classe 1 Nakamura}, a system of local holomorphic coordinates for $M_{t}$ is given by 
\[
\begin{cases}
    \zeta_{1}^{t} \doteq z_{1} + t_{11} \overline{z_{1}}, \\[3pt]
    \zeta_{2}^{t} \doteq z_{2} + t_{22} \overline{z_{2}} - \frac{t_{21}}{t_{11}} e^{-z_{1}} (e^{-t_{11} \overline{z_{1}}} - 1), \\[3pt]
    \zeta_{3}^{t} \doteq z_{3} + t_{33} \overline{z_{3}} + \frac{t_{31}}{t_{11}}e^{z_{1}} (e^{t_{11} \overline{z_{1}}} - 1).
\end{cases}
\]
Moreover, the universal cover of $M_{t}$ is $\C^{3}$, with $M_{t} = \Gamma_{t}\backslash C^{3}$ for $\Gamma_{t}$ defined by 
\[
\begin{cases}
    \zeta_{1}^{'} \doteq \zeta^{t}_{1} + \widetilde{\omega}_{1}, \\[3pt]
    \zeta_{2}^{'} \doteq e^{-\omega_{1}}\zeta^{t}_{2} + \widetilde{\omega}_{2} + \frac{t_{21}}{t_{11}} e^{- \zeta_{1}^{t} - \omega_{1}} (1 - e^{-t_{11} \overline{\omega_{1}}}), \\[3pt]
    \zeta_{3}^{'} \doteq e^{\omega_{1}}\zeta_{3}^{t} + \widetilde{\omega}_{3} - \frac{t_{31}}{t_{11}} e^{\zeta_{1}^{t} + \omega_{1}} (1 - e^{-t_{11} \overline{\omega_{1}}}),
\end{cases}
\]
where $\omega_{j} \in \Gamma$ for $j=1,2,3$ and $\widetilde{\omega}_{j} \doteq \omega_{j} + t_{jj} \overline{\omega_{j}}$. A set of $(1,0)$-forms which are invariant with respect to the action of $\Gamma_{t}$ is 
\[
\begin{cases}
    \varphi^{1}_{t} \doteq d\zeta_{1}^{t}, \\[3pt]
    \varphi^{2}_{t} \doteq e^{z_{1}} d\zeta_{2}^{t} - \frac{t_{21}}{t_{11}} e^{z_{1} - \zeta_{1}^{t}} d \zeta_{1}^{t}, \\[3pt]
    \varphi^{3}_{t} \doteq e^{-z_{1}} d \zeta_{3}^{t} - \frac{t_{31}}{t_{11}} e^{-z_{1} + \zeta_{1}^{t}} d \zeta_{1}^{t},
\end{cases}
\]
with structure equations 
\[
\begin{cases}
    d\varphi^{1}_{t} = 0, \\[3pt]
    d\varphi^{2}_{t} = \frac{1}{1-|t_{11}|^{2}} \varphi_{t}^{12} + \frac{t_{11}}{1 - |t_{11}|^{2}} \varphi_{t}^{2 \overline{1}}, \\[3pt]
    d\varphi^{3}_{t} = - \frac{1}{1-|t_{11}|^{2}} \varphi_{t}^{13} - \frac{t_{11}}{1-|t_{11}|^{2}} \varphi_{t}^{3 \overline{1}}.
\end{cases}
\]
In this case, we have
\[
\del \delbar(\varphi_{t}^{2 \overline{2}}) = -\frac{1}{1 - |t_{11}|^{2}} (|t_{11}|^{2} - 2 \mathfrak{Re}(t_{11}) + 1) \varphi_{t}^{12 \overline{12}},
\]
thus, there are no SKT metrics by Proposition \ref{Obstructions to the existence of p pluriclosed}. Moreover, $M_{t}$ admits a balanced metric; indeed, it is straightforward to verify that $d \omega_{t}^{2} = 0$, where $\omega_{t}$ is the fundamental form of the diagonal Hermitian metric on $M_{t}$.

In case \eqref{classe 3 Nakamura}, a system of local holomorphic coordinates is given by 
\[
\begin{cases}
    \zeta_{1}^{t} \doteq z_{1} - \text{log}(1 - t_{12} e^{z_{1}} \overline{z_{2}}), \\[3pt]
    \zeta_{2}^{t} \doteq z_{2} + t_{22} \overline{z_{2}}, \\[3pt]
    \zeta_{3}^{t} \doteq z_{3} + t_{33} \overline{z_{3}} + \frac{t_{32} e^{2 z_{1}} \overline{z_{2}}}{1- t_{12} e^{z_{1}} \overline{z_{2}}}.
\end{cases}
\]
In \cite{Nakamura75}, it is proved that the universal cover of $M_{t}$ is the set 
\[
W_{t} \doteq \{(\zeta_{1}, \zeta_{2}, \zeta_{3}) \in \C^{3} \; | \; (1 - |t_{22}|^{2}) e^{- \zeta_{1}} + t_{12} (\overline{\zeta_{2}} - \overline{t_{22}}\zeta_{2}) \neq 0\}
\]
and that $M_{t}$ can be written as $M_{t} = \Gamma_{t} \backslash W_{t}$, where $\Gamma_{t}$ is generated by
\[
\begin{cases}
    \zeta_{1}^{'} \doteq \zeta_{1} + \omega_{1} - \text{log}(1-t_{12} \overline{\omega_{2}} e^{\zeta_{1} + \omega_{1}}), \\[3pt]
    \zeta_{2}^{'} \doteq e^{-\omega_{1}}\zeta_{2} + \widetilde{\omega}_{2} , \\[3pt]
    \zeta_{3}^{'} \doteq e^{\omega_{1}}\zeta_{3} + \widetilde{\omega}_{3} + \frac{t_{32} \overline{\omega_{2}} e^{2 \zeta_{1} + 2 \omega_{1}}}{1 - t_{12} \overline{\omega_{2}} e^{\zeta_{1} + \omega_{1}}},
\end{cases}
\]
where $\omega_{j} \in \Gamma$ for $j=1,2,3$ and $\widetilde{\omega}_{j} \doteq \omega_{j} + t_{jj} \overline{\omega}_{j}$. Subsequently, K. Hasegawa in \cite{Has10} observed that in this case the complex structure on $M_{t}$ cannot be left-invariant. 

For this class, a set of $(1,0)$-forms invariant with respect to the action of $\Gamma_{t}$ is
\[
\begin{cases}
    \varphi^{1}_{t} \doteq \frac{1 - |t_{22}|^{2}}{A} d\zeta_{1}^{t}, \\[3pt]
    \varphi^{2}_{t} \doteq \frac{1 - |t_{22}|^{2}}{A} d\zeta_{2}^{t}, \\[3pt]
    \varphi^{3}_{t} \doteq \Big( e^{-\zeta_{1}^{t}} + \frac{t_{12}}{1 - |t_{22}|^{2}}(\overline{\zeta_{2}^{t}} - \overline{t_{22}} \zeta_{2}^{t}) \Big) \Big( d\zeta_{3}^{t} - \frac{t_{32}}{t_{12}} e^{\zeta_{1}^{t}} d\zeta_{1}^{t}\Big),
\end{cases}
\]
where $A=A(t,\zeta_{1}^{t}, \zeta_{2}^{t}, \overline{\zeta_{2}^{t}}) \doteq(1-|t_{22}|^{2}) e^{-\zeta_{1}^{t}} + t_{12} (\overline{\zeta_{2}^{t}} - \overline{t_{22} \zeta_{2}^{t}})$. The complex structure equations are 
\[
\begin{cases}
    d\varphi^{1}_{t} = -\frac{t_{12}\overline{t_{22}}}{1 - |t_{22}|^{2}} \varphi_{t}^{12} + \frac{t_{12}}{1 - |t_{22}|^{2}} \frac{\overline{A}}{A} \varphi_{t}^{1\overline{2}}, \\[3pt]
    d\varphi^{2}_{t} = \varphi_{t}^{12} + \frac{t_{12}}{1 - |t_{22}|^{2}} \frac{\overline{A}}{A} \varphi_{t}^{2\overline{2}}, \\[3pt]
    d\varphi^{3}_{t} = - \varphi_{t}^{13} - \frac{t_{12} \overline{t_{22}}}{1 - |t_{22}|^{2}} \varphi_{t}^{23} - \frac{t_{12}}{1 - |t_{22}|^{2}} \frac{\overline{A}}{A} \varphi_{t}^{3\overline{2}}. 
\end{cases}
\]
Since $\del \delbar (\varphi_{t}^{2\overline{2}}) = \varphi_{t}^{12 \overline{12}}$, by Proposition \ref{Obstructions to the existence of p pluriclosed} there are no SKT metrics along small deformations of the complex structure.

Finally, we observe that in this case the diagonal Hermitian metric is balanced. Indeed, a direct computation shows that 
\[ 
d \varphi_{t}^{1 \overline{1} 2 \overline{2}} = 0, \quad d \varphi_{t}^{1 \overline{1} 3 \overline{3}} = 0, \quad d \varphi_{t}^{2 \overline{2} 3 \overline{3}} = 0,
\]
hence $d \omega^{2}_{t} = 0$, where $\omega_{t}$ is the fundamental form of the diagonal Hermitian metric on $M_{t}$.

We mention that case \eqref{classe 4 Nakamura} is similar to case \eqref{classe 3 Nakamura}, while in case \eqref{classe 2 Nakamura} there is no easy obstruction to the existence of SKT metrics. In summary, we have proved the following theorem. 
\begin{theorem}\label{deformation of Nakamura}
Small deformations of the holomorphically parallelizable Nakamura manifold belonging to the classes \eqref{classe 1 Nakamura}, \eqref{classe 3 Nakamura} and \eqref{classe 4 Nakamura} admit balanced metrics, but do not admit any SKT metric.
\end{theorem}

\begin{remark}
    In agreement with the Fino–Vezzoni conjecture, the coexistence of balanced and SKT metrics is not possible for small deformations of the holomorphically parallelizable Nakamura manifold, even when the complex structure is not left-invariant.
\end{remark}

\section{The Kuranishi space of a 4-dimensional solvmanifold}\label{Section Kuranishi}

The goal of this section is to construct the Kuranishi space for the $4$-dimensional solvmanifold $M_\pi=\Gamma_\pi\backslash G$ defined in \cite[Section 6]{SferruzzaTomassini24} and to study if small deformations of $M_\pi$ admit SKT metrics. 

The manifold $M_\pi$ is defined starting from the complex Lie group $G=(\C^4,*)$ endowed with the operation $*:\C^4\times\C^4\to\C^4$ given by
\[
    ^t(y_1,y_2,y_3,y_4)*\, ^t(z_1,z_2,z_3,z_4):=\,^t(z_1+y_1,e^{-y_1}z_2+y_2,e^{y_1}z_3+y_3,z_4+y_4+\frac{1}{2}e^{y_1}y_2z_3-\frac{1}{2}e^{-y_1}y_3z_2).
\]
For the definition of the lattice $\Gamma_\pi$, see \cite[Section 6]{SferruzzaTomassini24}. A coframe of left-invariant holomorphic $(1,0)$-forms on $G$ is given in the standard coordinates $\{z_1,z_2,z_3,z_4\}$ of $\C^4$ by 
\begin{equation}\label{coframe4d}
    \varphi^1=dz_1, \quad \varphi^2=e^{z_1}dz_2, \quad \varphi^3=e^{-z_1}dz_3, \quad \varphi^4=dz_4-\frac{1}{2}z_2dz_3 + \frac{1}{2}z_3dz_2,
\end{equation}
and they satisfy the structure equations 
\begin{equation}\label{strequ4d} 
    d\varphi^1=0,\qquad d\varphi^2=\varphi^{12},\qquad d\varphi^3=-\varphi^{13},\qquad d\varphi^4=-\varphi^{23}.
\end{equation}
The frame $\{Z_1,Z_2,Z_3,Z_4\}$ dual to \eqref{coframe4d} is given in holomorphic coordinates by
\[
    Z_1=\frac{\del}{\del z_1}, \quad Z_2=e^{-z_1}\frac{\del}{\del z_2}-\frac{1}{2}e^{-z_1}z_3\frac{\del}{\del z_4}, \quad Z_3=e^{z_1}\frac{\del}{\del z_3}+ \frac{1}{2}e^{z_1}z_2\frac{\del}{\del z_4}, \quad Z_4=\frac{\del}{\del z_4}.
\]
Thanks to \cite[Corollary 6.2]{Kasuya14}, the Dolbeault cohomology group $H^{0,k}_{\delbar}(M_{\pi})$ can be computed by means of the subcomplex $B^{k}_{\Gamma_{\pi}} \subseteq\mathcal{A}^{0, k}(M_{\pi})$ defined by 
\begin{equation*}
    B^{k}_{\Gamma_{\pi}} \doteq \Big\langle \Big(\frac{\overline{\alpha_{I}}}{\alpha_{I}} \Big) \varphi^{\overline{I}} \;| \; I = \{i_{1}, \dots, i_{k}\} \subseteq \{1, \dots, 4\} \; \text{such that } \; \Big(\frac{\overline{\alpha_{I}}}{\alpha_{I}}\Big)|_{\Gamma_{\pi}} = 1  \Big\rangle,
\end{equation*}
where $\alpha_{I} = \alpha_{i_{1}} \cdot \dots \cdot \alpha_{i_{k}}$ and $\alpha_{1} \equiv \alpha_{4} \equiv 1$, $\alpha_{2}(y_{1}) = e^{-y_{1}}$, $\alpha_{3}(y_{1}) = e^{y_{1}}$ are the eigenvalues of the adjoint representation of $G$ and are characters of $\C$. Moreover, by \cite[Theorem 7.2]{Kasuya17} the Schouten bracket is compatible with $B^{k}_{\Gamma_{\pi}}$, hence we can compute the Kuranishi space using this subcomplex. 

As computed in \cite{SferruzzaTomassini24}, we have that the Dolbeault cohomology group $H^{0,1}_{\delbar}(M_\pi)$ is generated by 
\[
    \widetilde{\varphi}^{\overline{1}}\doteq\varphi^{\overline{1}}, \qquad\widetilde{\varphi}^{\overline{2}}\doteq e^{z_1-\overline{z}_1}\varphi^{\overline{2}}, \qquad\widetilde{\varphi}^{\overline{3}}\doteq e^{\overline{z}_1-z_1}\varphi^{\overline{3}}.
\]

A generic deformation for $M_\pi$ is given by $\psi(t)=\sum_{k=1}^\infty\psi_k(t)$, where $\psi_k(t)\in \mathcal{A}^{0,1}(T^{1,0}(M_\pi))$ is homogeneous of degree $k$ in $t$. Using the Maurer-Cartan equation $\delbar\psi(t)-\frac{1}{2}[\psi(t),\psi(t)]=0$, it is possible to express the terms $\psi_{k}(t)$. Since $\delbar\psi_1(t)=0$, the first term of the power expansion of $\psi$ is necessarily $\psi_1(t)=\sum_{\lambda=1}^3\sum_{i=1}^4t_{i\lambda}\widetilde{\varphi}^{\overline{\lambda}}\otimes Z_{i}$. 

The value of $\psi_2(t)$ is determined by the second term of the Maurer-Cartan equation $\delbar\psi_2(t)=\frac{1}{2}[\psi_1(t),\psi_1(t)]$. The right-hand side of this relation is computed thanks to the following relations (see also \cite[Page 97]{Nakamura75}):
\begin{equation}\label{bracketNakamura}
    \begin{aligned}
        &[\widetilde{\varphi}^{\overline{1}}\otimes Z_{\lambda}, \widetilde{\varphi}^{\overline{2}}\otimes Z_{\mu}]=\widetilde{\varphi}^{\overline{1}}\wedge\widetilde{\varphi}^{\overline{2}}\otimes([Z_{\lambda},Z_{\mu}]+\delta_{1\lambda}Z_{\mu});\\
        &[\widetilde{\varphi}^{\overline{1}}\otimes Z_{\lambda}, \widetilde{\varphi}^{\overline{3}}\otimes Z_{\mu}]=\widetilde{\varphi}^{\overline{1}}\wedge\widetilde{\varphi}^{\overline{3}}\otimes([Z_{\lambda},Z_{\mu}]-\delta_{1\lambda}Z_{\mu});\\
        &[\widetilde{\varphi}^{\overline{2}}\otimes Z_{\lambda}, \widetilde{\varphi}^{\overline{3}}\otimes Z_{\mu}]=\widetilde{\varphi}^{\overline{2}}\wedge\widetilde{\varphi}^{\overline{3}}\otimes([Z_{\lambda},Z_{\mu}]-\delta_{1\lambda}Z_{\mu}-\delta_{1\mu}Z_{\lambda}).
    \end{aligned}
\end{equation}
We can write $\frac{1}{2}[\psi_1(t),\psi_1(t)]=\sum_{j=1}^4\eta_j\otimes Z_j$, where 
\begin{align*}
    \eta_1&=t_{11}t_{12}\widetilde{\varphi}^{\overline{1}\overline{2}}-t_{11}t_{13}\widetilde{\varphi}^{\overline{1}\overline{3}}-2t_{12}t_{13}\widetilde{\varphi}^{\overline{2}\overline{3}}, \\
    \eta_2&=t_{12}t_{21}\widetilde{\varphi}^{\overline{1}\overline{2}}+(t_{21}t_{13}-2t_{11}t_{23})\widetilde{\varphi}^{\overline{1}\overline{3}}-2t_{12}t_{23}\widetilde{\varphi}^{\overline{2}\overline{3}}, \\
    \eta_3&=(2t_{11}t_{32}-t_{31}t_{12})\widetilde{\varphi}^{\overline{1}\overline{2}}-t_{13}t_{31}\widetilde{\varphi}^{\overline{1}\overline{3}}-2t_{32}t_{13}\widetilde{\varphi}^{\overline{2}\overline{3}}, \\
    \eta_4&=(t_{11}t_{42}+t_{21}t_{32}-t_{31}t_{22})\widetilde{\varphi}^{\overline{1}\overline{2}}+(t_{21}t_{33}-t_{11}t_{43}-t_{31}t_{23})\widetilde{\varphi}^{\overline{1}\overline{3}}+(t_{22}t_{33}-t_{12}t_{43}-t_{13}t_{42}-t_{23}t_{32})\widetilde{\varphi}^{\overline{2}\overline{3}}.
\end{align*}
By the Maurer-Cartan equation, $\delbar\psi_2(t)=\sum_{i=1}^4\eta_i\otimes Z_i$, hence each $\eta_i$ has zero Dolbeault cohomology. Since the forms $\widetilde{\varphi}^{\overline{1}\overline{2}}$ and $\widetilde{\varphi}^{\overline{1}\overline{3}}$ are Dolbeault harmonic with respect to the diagonal metric defined by the $\widetilde{\varphi}^j$'s, we necessarily have that the coefficients of those forms in $\eta_i$ are zero. We then have the following relations:
\begin{equation}\label{condizioni4d}
    \begin{aligned}
        t_{11}t_{12}&=t_{11}t_{13}=t_{12}t_{21}=(t_{21}t_{13}-2t_{11}t_{23})=(2t_{11}t_{32}-t_{31}t_{12})=t_{13}t_{31}\\[3pt]
        &=(t_{11}t_{42}+t_{21}t_{32}-t_{31}t_{22})=(t_{21}t_{33}-t_{11}t_{43}-t_{31}t_{23})=0.
    \end{aligned}
\end{equation}
The form $\widetilde{\varphi}^{\overline{2}\overline{3}}=\varphi^{\overline{2}\overline{3}}$ is $\delbar$-exact by \eqref{strequ4d}, hence the second term of the deformation $\psi_2(t)$ is given by
\begin{equation}\label{psi2prima}
    \psi_2(t)=\varphi^{\overline{4}}\otimes[2t_{12}t_{13}Z_1+2t_{12}t_{23}Z_2+2t_{32}t_{13}Z_3+(-t_{22}t_{33}+t_{12}t_{43}+t_{13}t_{42}+t_{23}t_{32})Z_4].
\end{equation}

The term $\psi_3(t)$ is determined by the bracket $[\psi_1(t),\psi_2(t)]$. Analogously to equations \eqref{bracketNakamura}, this is computed using the following relations:
\begin{equation}
    \begin{aligned}
        &[\widetilde{\varphi}^{\overline{2}}\otimes Z_{\lambda}, \varphi^{\overline{4}}\otimes Z_{1}]=\widetilde{\varphi}^{\overline{2}}\wedge\varphi^{\overline{4}}\otimes(-Z_{\lambda}-[Z_1,Z_{\lambda}]);\\
        &[\widetilde{\varphi}^{\overline{3}}\otimes Z_{\lambda}, \varphi^{\overline{4}}\otimes Z_{1}]=\widetilde{\varphi}^{\overline{3}}\wedge\varphi^{\overline{4}}\otimes(Z_{\lambda}-[Z_1,Z_{\lambda}]);\\
        &[\widetilde{\varphi}^{\overline{\alpha}}\otimes Z_{\lambda}, \varphi^{\overline{4}}\otimes Z_{\mu}]=\widetilde{\varphi}^{\overline{\alpha}}\wedge\varphi^{\overline{4}}\otimes([Z_{\lambda},Z_{\mu}])\quad\text{ for }\quad(\alpha,\mu)\notin\{(2,1), (3,1)\}.
    \end{aligned}
\end{equation}
Thus, the value of bracket $[\psi_1(t),\psi_2(t)]$ is the following:
\begin{align*}
    [\psi_1(t),\psi_2(t)]=&\widetilde{\varphi}^{\overline{2}}\wedge\varphi^{\overline{4}}\otimes[-2t_{12}^2t_{13}Z_1-2t_{12}^2t_{23}Z_2-2t_{12}t_{13}t_{32}Z_3+(2t_{13}t_{22}t_{32}-2t_{2}t_{12}t_{13}4-2t_{12}t_{23}t_{32})Z_4]\\[3pt]
    &+\widetilde{\varphi}^{\overline{2}}\wedge\varphi^{\overline{4}}\otimes[2t_{12}t_{13}^2Z_1+2t_{12}t_{13}t_{23}Z_2+2t_{13}^2t_{32}Z_3+(2t_{12}t_{13}t_{43}-2t_{12}t_{23}t_{33}+2t_{13}t_{23}t_{32})Z_4].
\end{align*}
Since the forms $\widetilde{\varphi}^{\overline{2}}\wedge\varphi^{\overline{4}}$ and $\widetilde{\varphi}^{\overline{3}}\wedge\varphi^{\overline{4}}$ are Dolbeault harmonic with respect to the diagonal metric defined by the $\widetilde{\varphi}^j$'s, their coefficients must vanish. Thus, $\psi_3(t)=0$ and we additionally have the vanishing of the following terms:
\begin{equation}\label{condizioniBracket12}
    t_{12}t_{13}=t_{12}t_{23}=t_{13}t_{32}=0.
\end{equation}
By \eqref{psi2prima} and \eqref{condizioniBracket12}, the second term of the deformation is 
\begin{equation}\label{psi2}
    \psi_2(t)=\varphi^{\overline{4}}\otimes(-t_{22}t_{33}+t_{12}t_{43}+t_{13}t_{42}+t_{23}t_{32})Z_4.
\end{equation}
Finally, since $[\varphi^{\overline{4}}\otimes Z_4, \varphi^{\overline{4}}\otimes Z_4]$ vanishes trivially, we have $[\psi_2(t),\psi_2(t)]=0$. Thus, the terms of the deformation of degree $\ge 3$ are zero. 

In summary, a generic deformation for $M_{\pi}$ is given by
\begin{equation}
    \psi(t)=\psi_1(t)+\psi_2(t)=\sum_{\lambda=1}^3\sum_{i=1}^4t_{i\lambda}\widetilde{\varphi}^{\overline{\lambda}}\otimes Z_{i}+(-t_{22}t_{33}+t_{12}t_{43}+t_{13}t_{42}+t_{23}t_{32})\varphi^{\overline{4}}\otimes Z_4,
\end{equation}
with the following relations on the parameters:
\begin{equation}\label{condizioniFinali4d}
    \begin{aligned}
        &t_{11}t_{12}=t_{11}t_{13}=t_{12}t_{21}=(t_{21}t_{13}-2t_{11}t_{23})=(2t_{11}t_{32}-t_{31}t_{12})=t_{13}t_{31}\\[3pt]
        &=t_{12}t_{13}=t_{12}t_{23}=t_{13}t_{32}=(t_{11}t_{42}+t_{21}t_{32}-t_{31}t_{22})=(t_{21}t_{33}-t_{11}t_{43}-t_{31}t_{23})=0.
    \end{aligned}
\end{equation}
We can now consider $4$ different classes of deformations, depending on the parameters we assume to vanish:
\begin{enumerate}[label=(\roman*)]
    \setlength{\itemsep}{.4em}
    \item\label{class14d} $t_{11}\neq0$. By \eqref{condizioniFinali4d}, this implies $t_{12}=t_{13}=t_{32}=t_{23}=0$, $t_{11}t_{42}=t_{31}t_{22}$, $t_{11}t_{43}=t_{21}t_{33}$;
    \item\label{class24d} $t_{11}=t_{12}=t_{13}=0$, $t_{21}t_{32}=t_{31}t_{22}$, $t_{21}t_{33}=t_{31}t_{23}$;
    \item\label{class34d} $t_{12}\neq0$. By \eqref{condizioniFinali4d}, this implies $t_{11}=t_{21}=t_{31}=t_{13}=t_{23}=0$;
    \item\label{class44d} $t_{13}\neq0$. By \eqref{condizioniFinali4d}, this implies $t_{11}=t_{12}=t_{21}=t_{31}=t_{32}=0$.
\end{enumerate}

Our goal is to establish whether the classes of deformations \ref{class14d} and \ref{class24d} admit any SKT metric. We begin with class \ref{class14d}. By the conditions on the parameters, the deformation in this case is the following:
\[
    \psi(t)=t_{11}\widetilde{\varphi}^{\overline{1}}\otimes Z_1+(t_{21}\widetilde{\varphi}^{\overline{1}}+t_{22}\widetilde{\varphi}^{\overline{2}})\otimes Z_2+(t_{31}\widetilde{\varphi}^{\overline{1}}+t_{33}\widetilde{\varphi}^{\overline{3}})\otimes Z_3+(t_{41}\widetilde{\varphi}^{\overline{1}}+t_{42}\widetilde{\varphi}^{\overline{2}}+t_{43}\widetilde{\varphi}^{\overline{3}}-t_{22}t_{33}\varphi^{\overline{4}})\otimes Z_4.
\]
Thus, by Lemma \ref{Isomorfismo tra lo spazio delle forme}, a coframe of $(1,0)$-forms for $M_{t}$ is given by $\{\varphi^{1}_{t}, \dots, \varphi^{4}_{t}\}$ where $\varphi^{j}_{t} \doteq e^{\iota_{\varphi(t)}| \iota_{\overline{\varphi(t)}}}(\varphi^{j})$; more precisely,
\[
\begin{cases}
    \varphi^{1}_{t} =  \varphi^{1} + t_{11} \varphi^{\overline{1}},\\[3pt]
    \varphi^{2}_{t} = \varphi^{2} + t_{21} \varphi^{\overline{1}} + t_{22} e^{z_{1} - \overline{z_{1}}} \varphi^{\overline{2}},\\[3pt]
    \varphi^{3}_{t} = \varphi^{3} + t_{31} \varphi^{\overline{1}} + t_{33} e^{\overline{z_{1}} - z_{1}} \varphi^{\overline{3}},\\[3pt]
    \varphi^{4}_{t} = \varphi^{4} + t_{41} \varphi^{\overline{1}}+t_{42}e^{z_{1} - \overline{z_{1}}}\varphi^{\overline{2}}+t_{43}e^{\overline{z_{1}} - z_{1}}\varphi^{\overline{3}}-t_{22}t_{33}\varphi^{\overline{4}}.
\end{cases}
\]
By these relations, the differentials of the elements of the $(1,0)$-coframe are
\[
\begin{cases}
    d \varphi^{1}_{t} = 0,\\[3pt]
    d \varphi^{2}_{t} = \varphi^{12} + t_{22} e^{z_{1} - \overline{z_{1}}} \varphi^{1 \overline{2}}, \\[3pt]
    d \varphi^{3}_{t} = -\varphi^{13} - t_{33} e^{\overline{z_{1}} - z_{1}} \varphi^{1 \overline{3}}, \\[3pt]
    d \varphi^{4}_{t} = - \varphi^{23}+ t_{42} e^{z_{1} - \overline{z_{1}}} \varphi^{1 \overline{2}}- t_{43} e^{\overline{z_{1}} - z_{1}} \varphi^{1 \overline{3}}+t_{22}t_{33}\varphi^{\overline{2}\overline{3}}.
\end{cases}
\]
Hence, to compute the structure equations, we need to express $\varphi^1$, $ \varphi^2$ and $ \varphi^3$ in terms of $\{\varphi^{1}_{t}, \varphi^{\overline{1}}_{t}, \varphi^{2}_{t}, \varphi^{\overline{2}}_{t}, \varphi^{3}_{t}, \varphi^{\overline{3}}_{t}\}$. It is easy to see that
\[
\begin{cases}
    \varphi^{1} = \frac{1}{1-|t_{11}|^{2}} (\varphi^{1}_{t} - t_{11} \varphi_{t}^{\overline{1}}), \\[3pt]
    \varphi^{2} = \frac{1}{1-|t_{22}|^{2}} \Big(\varphi^{2}_{t} - t_{22} e^{z_{1} - \overline{z_{1}}} \varphi_{t}^{\overline{2}} + \Big(\frac{t_{21}\overline{t_{11}} + \overline{t_{21}}t_{22} e^{z_{1}-\overline{z_{1}}}}{1-|t_{11}|^{2}}\Big) \varphi^{1}_{t} - \Big(\frac{t_{21} + t_{11}\overline{t_{21}}t_{22} e^{z_{1}-\overline{z_{1}}}}{1-|t_{11}|^{2}}\Big) \varphi^{\overline{1}}_{t} \Big),\\[3pt]
    \varphi^{3} = \frac{1}{1-|t_{33}|^{2}} \Big( \varphi^{3}_{t} - t_{33} e^{\overline{z_{1}} - z_{1}} \varphi_{t}^{\overline{3}} + \Big(\frac{t_{31}\overline{t_{11}} + \overline{t_{31}}t_{33} e^{\overline{z_{1}} - z_{1}}}{1-|t_{11}|^{2}}\Big) \varphi^{1}_{t} - \Big(\frac{t_{31} + t_{11}\overline{t_{31}}t_{33} e^{\overline{z_{1}} -z_{1}}}{1-|t_{11}|^{2}}\Big) \varphi^{\overline{1}}_{t} \Big).
\end{cases}
\]
We define $T_1\doteq\frac{1}{1-|t_{11}|^2}$, $T_2\doteq\frac{1}{1-|t_{22}|^2}$, $T_3\doteq\frac{1}{1-|t_{33}|^2}$, $A(t)\doteq|t_{11}|^2-|t_{22}|^2+|t_{22}|^2|t_{33}|^2-|t_{11}|^2|t_{22}|^2|t_{33}|^2$ and $E\doteq e^{z_{1} - \overline{z_{1}}}$. By direct computations, it is possible to show that the structure equations for the class of deformation defined by \ref{class14d} are the following:
\[
\begin{aligned}
    d \varphi^{1}_{t} =& 0,\\[3pt]
    d \varphi^{2}_{t} =& T_{1} \Big( \varphi_{t}^{12} + t_{11} \varphi^{2 \overline{1}}_{t} - t_{21} \varphi_{t}^{1 \overline{1}} \Big), \\[3pt]
    d \varphi^{3}_{t} =& - T_{1} \Big(\varphi_{t}^{13} + t_{11} \varphi^{3 \overline{1}}_{t} - t_{31} \varphi_{t}^{1 \overline{1}} \Big), \\[3pt]
    d \varphi^{4}_{t} =& T_1\Big(T_2(Et_{22}t_{31}\overline{t_{21}}+t_{42}t_{21}\overline{t_{22}})-T_3(\overline{E}t_{21}t_{33}\overline{t_{31}}+t_{43}t_{31}\overline{t_{33}})\Big)\varphi_t^{1\overline{1}}\\[3pt]
    &+\frac{T_1T_3}{t_{11}}\Big(T_2t_{31}A(t)+\overline{E}t_{33}t_{11}\overline{t_{31}}\Big)\varphi_t^{12}-\frac{T_1T_2}{t_{11}}\Big(T_3t_{21}A(t)+Et_{22}t_{11}\overline{t_{21}}\Big)\varphi_t^{13}\\[3pt]
    &+T_1T_3\Big(t_{31}+\overline{E}t_{11}t_{33}\overline{t_{31}}\Big)\varphi_t^{2\overline{1}}-T_1T_2\Big(t_{21}+Et_{11}t_{22}\overline{t_{21}}\Big)\varphi_t^{3\overline{1}}+T_2Et_{42}\varphi_t^{1\overline{2}}-T_3\overline{E}t_{43}\varphi_t^{1\overline{3}}\\[3pt]
    &+T_2T_3\Big(|t_{22}|^2|t_{33}|^2-1\Big)\varphi_t^{23}+T_1T_3\overline{E}t_{33}\varphi_t^{2\overline{3}}-T_1T_2Et_{22}\varphi_t^{3\overline{2}}.
    \end{aligned}
\]
We now study the existence of SKT and astheno-K\"ahler metrics for class \ref{class14d}. Since
\begin{align*}
    \ddbar\varphi_t^{1\overline{1}4\overline{4}}=&\Big(T_1^2T_3^2|t_{33}|^2+T_1^2T_2^2|t_{22}|^2+T_2^2T_3^2(|t_{22}|^2|t_{33}|^2-1)^2\Big)\varphi_t^{1\overline{1}2\overline{2}3\overline{3}},\\[3pt]
    \ddbar\varphi_{t}^{2 \overline{2}} =&  T_{1}^{2} (|t_{11}|^{2} - 2 \mathfrak{Re}(t_{11}) + 1) \varphi_{t}^{1 \overline{1} 2\overline{2}},
\end{align*}
by Proposition \ref{Obstructions to the existence of p pluriclosed} there are no SKT and astheno-K\"ahler metrics for small deformations.

\medskip

For class \ref{class24d}, by the conditions on the parameters, the deformation is the following:
\[
    \psi(t)= (t_{21}\widetilde{\varphi}^{\overline{1}}+t_{22}\widetilde{\varphi}^{\overline{2}} + t_{23} \widetilde{\varphi}^{\overline{3}})\otimes Z_2+ (t_{31}\widetilde{\varphi}^{\overline{1}} + t_{32} \widetilde{\varphi}^{\overline{2}} +t_{33}\widetilde{\varphi}^{\overline{3}})\otimes Z_3+(t_{41}\widetilde{\varphi}^{\overline{1}}+t_{42}\widetilde{\varphi}^{\overline{2}}+t_{43}\widetilde{\varphi}^{\overline{3}} - (t_{22}t_{33} - t_{23} t_{32}) \varphi^{\overline{4}})\otimes Z_4.
\]
Thus, working in the same way as above, a coframe of $(1,0)$-forms for $M_{t}$ is given by 
\[
\begin{cases}
    \varphi^{1}_{t} = \varphi^{1}, \\[3pt]
    \varphi^{2}_{t} = \varphi^{2} + t_{21} \varphi^{\overline{1}} + t_{22} e^{z_{1} - \overline{z_{1}}} \varphi^{\overline{2}} + t_{23} e^{\overline{z_{1}} - z_{1} } \varphi^{\overline{3}}, \\[3pt]
    \varphi^{3}_{t} = \varphi^{3} + t_{31} \varphi^{\overline{1}} + t_{32} e^{ z_{1} - \overline{z_{1}} } \varphi^{\overline{2}} + t_{33} e^{\overline{z_{1}} - z_{1}} \varphi^{\overline{3}}, \\[3pt]
    \varphi_{t}^{4} = \varphi^{4} + t_{41} \varphi^{\overline{1}} + t_{42} e^{z_{1} - \overline{z_{1}}} \varphi^{\overline{2}} + t_{43} e^{\overline{z_{1}}-z_{1}} \varphi^{\overline{3}} - (t_{22}t_{33} - t_{23} t_{32}) \varphi^{\overline{4}},
\end{cases}
\]
and
\begin{equation}\label{Differentiale per la seconda deformazione}
    \begin{cases}
    d \varphi^{1}_{t} = 0, \\[3pt]
    d\varphi^{2}_{t} = \varphi^{12} + t_{22} e^{z_{1} - \overline{z_{1}}} \varphi^{1 \overline{2}} - t_{23} e^{\overline{z_{1}} - z_{1} } \varphi^{1 \overline{3}}, \\[3pt]
    d \varphi^{3}_{t} = - \varphi^{13} + t_{32} e^{z_{1} - \overline{z_{1}}} \varphi^{1 \overline{2}} - t_{33} e^{\overline{z_{1}} - z_{1} } \varphi^{1 \overline{3}}, \\[3pt]
    d \varphi^{4}_{t} = - \varphi^{23} + t_{42} e^{z_{1} - \overline{z_{1}}} \varphi^{1 \overline{2}} - t_{43} e^{\overline{z_{1}} - z_{1} } \varphi^{1 \overline{3}} + (t_{22}t_{33} - t_{23} t_{32}) \varphi^{\overline{23}}.
\end{cases}
\end{equation}
By combining formulas \eqref{Differentiale per la seconda deformazione} and \eqref{eq:phi23}, \eqref{eq:phi12}, \eqref{eq:phi13}, \eqref{eq:phi1 2bar}, \eqref{eq:phi1 3bar}, we get 
\begin{equation*}
    \begin{split}
        d \varphi^{2}_{t}  = &  \Big(1 + \frac{2}{f} t_{23}\big(t_{32} \overline{\lambda} + \overline{E}^{2} \overline{t_{32}}\big) \Big)\varphi_{t}^{12} + \frac{2}{f} t_{23} \big(\overline{t_{33}} - t_{22} \overline{\lambda} \big) \varphi_{t}^{13} - \frac{2}{f} E t_{23}\big(t_{32}\overline{t_{33}} + \overline{E}^{2} t_{22} \overline{t_{32}}\big)\varphi_{t}^{1 \overline{2}} \\
        & - \frac{2}{f} \overline{E} t_{23} \big(1 - |t_{22}|^{2} - E^{2}t_{32} \overline{t_{23}} \big) \varphi_{t}^{1\overline{3}} - t_{21} \Big(1 + \frac{2}{f}\big(|t_{33}|^{2} + \overline{E}^{2}t_{23}\overline{t_{32}} \big) \Big) \varphi_{t}^{1 \overline{1}}  ,\\
        d \varphi_{t}^{3} = &  -\Big( 1 + \frac{2}{f} t_{32}\big(t_{23} \overline{\lambda} + E^{2}\overline{t_{23}}\big) \Big)\varphi_{t}^{13} - \frac{2}{f} t_{32} \big(\overline{t_{22}} - t_{33} \overline{\lambda} \big)  \varphi_{t}^{12} + \frac{2}{f} \overline{E} t_{32}\big(t_{23}\overline{t_{22}} + E^{2} t_{33} \overline{t_{23}}\big)\varphi_{t}^{1 \overline{3}} \\
        & + \frac{2}{f} E t_{32} \big(1 - |t_{33}|^{2} - \overline{E}^{2}t_{23} \overline{t_{32}} \big) \varphi_{t}^{1\overline{2}} + t_{31} \Big( 1 + \frac{2}{f} \big(|t_{22}|^{2} + E^{2}t_{32} \overline{t_{23}}\big) \Big) \varphi_{t}^{1 \overline{1}} ,
    \end{split}
\end{equation*}
and 
\begin{equation*}
    \begin{split}
        d \varphi_{t}^{4} = & - \frac{1}{f} \big(1 - |\lambda|^{2}\big) \varphi_{t}^{23} + \frac{1}{f}\big(E t_{32} + \overline{Et_{32}}\lambda \big)\varphi_{t}^{2 \overline{2}} - \frac{1}{f} \big(\overline{E}t_{23}  + E \overline{t_{23}} \lambda\big) \varphi_{t}^{3 \overline{3}} + \frac{1}{f}t_{31} \varphi_{t}^{2 \overline{1}} - \frac{1}{f} t_{21} \varphi_{t}^{3 \overline{1}} \\[3pt]
        & + \frac{1}{f}\Big(t_{42} \big(t_{21} \overline{t_{22}} + E^{2} t_{31} \overline{t_{23}} \big) - t_{43} \overline{E}^{2} \big( t_{21}\overline{t_{32}} + E^{2}t_{31} \overline{t_{33}} \big) \Big) \varphi_{t}^{1\overline{1}}  + \frac{1}{f}\overline{E}\big(t_{33} - \overline{t_{22}} \lambda\big) \varphi_{t}^{2 \overline{3}} - \frac{1}{f}E \big(t_{22} - \overline{t_{33}}\lambda \big) \varphi_{t}^{3 \overline{2}} \\[3pt]
        & + \frac{1}{f}\Big( E t_{32}\overline{t_{21}} + \overline{E}t_{33}\overline{t_{31}} + \overline{E}^{2}t_{43}\big( \overline{t_{32}} - E^{2} \overline{t_{23}} |t_{32}|^{2} +t_{32} \overline{t_{22}} \overline{t_{33}} \big)- t_{42}\big(\overline{t_{22}}(1 - |t_{33}|^{2}) + t_{22} \overline{t_{23}} \overline{t_{32}} \big) \Big) \varphi_{t}^{12} \\[3pt]
        & - \frac{1}{f} \Big(E t_{22}\overline{t_{21}} + \overline{E}t_{23}\overline{t_{31}} + E^{2}t_{42} \big(\overline{t_{23}} - \overline{t_{32}}|t_{23}|^{2}  + t_{23} \overline{t_{22}} \overline{t_{33}} \big) - t_{43} \big( \overline{t_{33}} (1 - |t_{22}|^{2}) + t_{22} \overline{t_{23}} \overline{t_{32}} \big)  \Big)\varphi_{t}^{13} \\[3pt]
        & + \frac{1}{f}\Big( t_{42} E \big(1 - |t_{33}|^{2} - \overline{E}^{2} t_{23} \overline{t_{32}} \big) - t_{43} \overline{E} \big(t_{22} \overline{t_{32}} + E^{2} t_{32} \overline{t_{33}} \big)  \Big)\varphi_{t}^{1 \overline{2}} \\[3pt]
        & + \frac{1}{f}\Big(t_{42} E \big(t_{33} \overline{t_{23}} + \overline{E}^{2} t_{23} \overline{t_{22}} \big) - t_{43} \overline{E} \big(1 - |t_{22}|^{2}  - E^{2} t_{32} \overline{t_{23}}\big)\Big)\varphi_{t}^{1 \overline{3}} ,
    \end{split}
\end{equation*}
where $f \doteq 1 - |t_{22}|^{2} - |t_{33}|^{2} - E^{2}t_{32} \overline{t_{23}} - \overline{E}^{2}t_{23} \overline{t_{32}} + |t_{22}t_{33} -t_{23}t_{32}|^{2}$, $E = e^{z_{1} - \overline{z_{1}}}$ and $\lambda \doteq t_{22}t_{33} - t_{23} t_{32}$.

As in case \ref{class14d}, there are no SKT metrics along small deformations of the complex structure. Indeed, a direct computation shows that 
\begin{equation*}
    \del \delbar ( \varphi_{t}^{1 \overline{1} 4 \overline{4}}) = \frac{1}{f^{2}} \big(2 (1 - |\lambda|^{2})^{2} - f (1 + |\lambda|^{2}) \big) \varphi_{t}^{1 \overline{1} 2 \overline{2} 3 \overline{3}}
\end{equation*}
hence, the thesis follows by Proposition \ref{Obstructions to the existence of p pluriclosed}.

Due to the computational complexity for cases \ref{class34d} and \ref{class44d}, we do not address these two classes of deformations here.

To summarize, we have proved the following theorem. 

\begin{theorem}\label{teorema 4d}
    Small deformations of the solvmanifold $M_\pi=\Gamma_\pi\backslash G$ defined in \cite[Section 6]{SferruzzaTomassini24} belonging to classes \ref{class14d} and \ref{class24d} do not admit any SKT metric. Moreover, small deformations belonging to class \ref{class14d} do not admit any astheno-K\"ahler metric.
\end{theorem}

\section{Appendix}\label{Appendice}
The aim of this Appendix is to provide useful formulas for the small deformations of class \ref{class24d} of Section \ref{Section Kuranishi}. 
We recall that a coframe of $(1,0)$-forms for $M_{t}$ as in class \ref{class24d} is given by 
\[
\begin{cases}
    \varphi^{1}_{t} = \varphi^{1}, \\[3pt]
    \varphi^{2}_{t} = \varphi^{2} + t_{21} \varphi^{\overline{1}} + t_{22} e^{z_{1} - \overline{z_{1}}} \varphi^{\overline{2}} + t_{23} e^{\overline{z_{1}} - z_{1} } \varphi^{\overline{3}}, \\[3pt]
    \varphi^{3}_{t} = \varphi^{3} + t_{31} \varphi^{\overline{1}} + t_{32} e^{ z_{1} - \overline{z_{1}} } \varphi^{\overline{2}} + t_{33} e^{\overline{z_{1}} - z_{1}} \varphi^{\overline{3}}, \\[3pt]
    \varphi_{t}^{4} = \varphi^{4} + t_{41} \varphi^{\overline{1}} + t_{42} e^{z_{1} - \overline{z_{1}}} \varphi^{\overline{2}} + t_{43} e^{\overline{z_{1}}-z_{1}} \varphi^{\overline{3}} - (t_{22}t_{33} - t_{23} t_{32}) \varphi^{\overline{4}}.
\end{cases}
\]
It is straightforward to check that 
\[
\varphi^{2} = S_{2} \Big( E t_{22} \overline{t_{21}} \varphi^{1}_{t} - t_{21} \varphi_{t}^{\overline{1}} + \varphi^{2}_{t} - t_{22} E \varphi^{\overline{2}}_{t} + E^{2} t_{22} \overline{t_{23}} \varphi^{3} - t_{23} \overline{E} \varphi^{\overline{3}} \Big),
\]
where $S_{2} \doteq \frac{1}{1 - |t_{22}|^{2}}$, $E = e^{z_{1} - \overline{z_{1}}}$ and
\[
\begin{split}
    \varphi^{3} & = - t_{31} \varphi_{t}^{\overline{1}} - t_{32} E \varphi^{\overline{2}} + \varphi_{t}^{3} - t_{33} \overline{E} \varphi^{\overline{3}}  \\[3pt]
    & = S_{2} E t_{32} \overline{t_{21}} \varphi_{t}^{1} - S_{2} t_{31} \varphi^{\overline{1}}_{t} + S_{2} t_{32} \overline{t_{22}} \varphi_{t}^{2} - S_{2} E t_{32} \varphi_{t}^{\overline{2}} + \varphi_{t}^{3} + S_{2} E^{2} t_{32} \overline{t_{23}} \varphi^{3} - \overline{E} B(t) \varphi^{\overline{3}},
\end{split}
\]
where $B(t) \doteq t_{33} + S_{2} t_{32}  t_{23}\overline{t_{22}}$. By direct computation, we obtain the following
\begin{equation*}\label{Espressione di varphi3}
    \begin{split}
        \varphi^{3} = &  S_{3} \Big( \sigma^{3}_{1} \varphi_{t}^{1} - \sigma^{3}_{\overline{1}} \varphi_{t}^{\overline{1}} + \sigma^{3}_{2} \varphi_{t}^{2} - \sigma^{3}_{\overline{2}} \varphi_{t}^{\overline{2}} + \varphi^{3}_{t} - \sigma_{\overline{3}}^{3} \varphi_{t}^{\overline{3}}  \Big),
    \end{split}
\end{equation*}
where $S_{3} \doteq \frac{S}{1 - |S B(t)|^{2}}$, $S \doteq \frac{1}{1- S_{2}E^{2} t_{32} \overline{t_{23}}}$ and 
\begin{equation*}
    \begin{split}
        \sigma_{1}^{3} \doteq & ES_{2}  t_{32} \overline{t_{21}} + \overline{E}\overline{S}S_{2}  B(t) \overline{t_{31}} = S_{2} \overline{S} E (t_{32} \overline{t_{21}} + \overline{E}^{2} t_{33} \overline{t_{31}}), \\[3pt]
        \sigma^{3}_{\overline{1}} \doteq & S_{2}t_{31} + \overline{E}^{2} \overline{S}S_{2} B(t) t_{21} \overline{t_{32}} = S_{2} \overline{S} t_{31}, \\[3pt]
        \sigma_{2}^{3} \doteq & S_{2} \big(t_{32} \overline{t_{22}} + \overline{E}^{2}\overline{S} B(t)\overline{t_{32}} \big) = S_{2}\overline{S}(t_{32} \overline{t_{22}} + \overline{E}^{2} t_{33}\overline{t_{32}}), \\[3pt]
        \sigma^{3}_{\overline{2}} \doteq & S_{2} \big(E t_{32} + \overline{ES} B(t) \overline{t_{32}} t_{22} \big) = S_{2} \overline{S}E(t_{32} - \overline{E}^{2} t_{23} |t_{32}|^{2} + \overline{E}^{2} t_{22}t_{33} \overline{t_{32}}), \\[3pt]
        \sigma_{\overline{3}}^{3} \doteq & \overline{ES}B(t) = S_{2}\overline{S}  \overline{E} \big(t_{33} (1 - |t_{22}|^{2}) + t_{32}t_{23} \overline{t_{22}} \big).
    \end{split}
\end{equation*}
Thus, we can write
\begin{align}\label{Espressione finale di varphi3}
    \varphi^{3} = &  \frac{1}{f} \Big(E \big(t_{32}\overline{t_{21}} + \overline{E}^{2} t_{33} \overline{t_{31}} \big) \varphi_{t}^{1} - t_{31} \varphi_{t}^{\overline{1}} + (t_{32} \overline{t_{22}} + \overline{E}^{2} t_{33}\overline{t_{32}}) \varphi_{t}^{2} - E(t_{32} - \overline{E}^{2} t_{23} |t_{32}|^{2} + \overline{E}^{2} t_{22}t_{33} \overline{t_{32}}) \varphi_{t}^{\overline{2}} \\
    & + \big(1 - |t_{22}|^{2} - \overline{E}^{2}t_{23} \overline{t_{32}} \big) \varphi^{3}_{t} -\overline{E} \big(t_{33} (1 - |t_{22}|^{2}) + t_{32}t_{23} \overline{t_{22}} \big)\varphi_{t}^{\overline{3}}  \Big),\notag
\end{align}
where $f \doteq \frac{1}{S_{2}S_{3} \overline{S}} = 1 - |t_{22}|^{2} - |t_{33}|^{2} - E^{2}t_{32} \overline{t_{23}} - \overline{E}^{2}t_{23} \overline{t_{32}} + |t_{22}t_{33} -t_{23}t_{32}|^{2}$. 

Now, we can express $\varphi^{2}$ in terms of $\{\varphi^{1}_{t}, \varphi^{\overline{1}}_{t}, \varphi^{2}_{t}, \varphi^{\overline{2}}_{t}, \varphi^{3}_{t}, \varphi^{\overline{3}}_{t}\}$ using \eqref{Espressione di varphi3}. Indeed, straightforward computations show that 
\begin{equation*}\label{Espressione di varphi2}
    \varphi_{2} =  S_{2} \Big( \sigma^{2}_{1}  \varphi^{1}_{t} -\sigma_{\overline{1}}^{2} \varphi_{t}^{\overline{1}} + \sigma_{2}^{2} \varphi_{t}^{2} - \sigma_{\overline{2}}^{2} \varphi_{t}^{\overline{2}} + \sigma_{3}^{2}  \varphi_{t}^{3} - \sigma_{\overline{3}}^{2} \varphi_{t}^{\overline{3}} \Big),
\end{equation*}
where 
\begin{align*}
        & \sigma_{1}^{2} \doteq S_{3} \overline{S} E (t_{22}\overline{t_{21}} + \overline{E}^{2} t_{23} \overline{t_{31}}), \quad && \sigma_{\overline{1}}^{2} \doteq  S_{3}\overline{S}t_{21}, \\[3pt]
        & \sigma_{2}^{2} \doteq  S_{3}\overline{S}(1-|t_{33}|^{2} - E^{2}t_{32}\overline{t_{23}}) , && \sigma_{\overline{2}}^{2} \doteq   S_{3}\overline{S}E\big(t_{22}(1-|t_{33}|^{2}) + t_{23}t_{32}\overline{t_{33}}\big), \\[3pt]
        & \sigma_{3}^{2} \doteq S_{3} \overline{S}(E^{2}t_{22}\overline{t_{23}} + t_{23} \overline{t_{33}}), \quad && \sigma_{\overline{3}}^{2} \doteq   S_{3} \overline{S} \overline{E}( t_{23} - E^{2}t_{32}|t_{23}|^{2} + E^{2} t_{22}t_{33} \overline{t_{23}}).
\end{align*}
Hence, 
\begin{align}\label{Espressione finale di varphi2}
        \varphi^{2} = & \frac{1}{f} \Big( E \big( t_{22}\overline{t_{21}} + \overline{E}^{2} t_{23} \overline{t_{31}}  \big) \varphi^{1}_{t} -t_{21} \varphi_{t}^{\overline{1}} + \big( 1 - |t_{33}|^{2} - E^{2}  t_{32} \overline{t_{23}} \big) \varphi_{t}^{2}  \\
        & - E\big(t_{22}(1 - |t_{33}|^{2}) + t_{23} t_{32} \overline{t_{33}}\big) \varphi_{t}^{\overline{2}} + \big(t_{23}\overline{t_{33}} + E^{2} t_{22}\overline{t_{23}}\big) \varphi_{t}^{3} - \overline{E} ( t_{23} - E^{2}t_{32}|t_{23}|^{2} + E^{2} t_{22}t_{33} \overline{t_{23}}) \varphi_{t}^{\overline{3}} \Big).\notag
\end{align}

Denote by $\lambda \doteq t_{22}t_{33} - t_{23} t_{32}$. By \eqref{Espressione finale di varphi3}, \eqref{Espressione finale di varphi2}, we get 

\begin{align}
    \varphi^{23} = & \frac{1}{f} \Big( - E \big(t_{32} \overline{t_{21}} + \overline{E}^{2} t_{33} \overline{t_{31}} \big)  \varphi_{t}^{12} + E \big(t_{22} \overline{t_{21}} + \overline{E}^{2} t_{23} \overline{t_{31}} \big) \varphi_{t}^{13} + \varphi_{t}^{23} + \overline{t_{31}}\lambda \varphi_{t}^{1 \overline{2}} - \overline{t_{21}} \lambda \varphi_{t}^{1 \overline{3}} \label{eq:phi23} \\
    & - E t_{32} \varphi_{t}^{2 \overline{2}} + \overline{E} t_{23} \varphi_{t}^{3 \overline{3}}  + t_{21} \varphi_{t}^{3 \overline{1}} - t_{31} \varphi_{t}^{2 \overline{1}} - \overline{E} t_{33} \varphi_{t}^{2 \overline{3}}  + E t_{22} \varphi_{t}^{3 \overline{2}} + \overline{\lambda} \varphi_{t}^{\overline{23}}\Big), \notag \\
    \varphi^{12} = & \frac{1}{f} \Big( - t_{21} \varphi_{t}^{1\overline{1}} + \big( 1 - |t_{33}|^{2} - E^{2}  t_{32} \overline{t_{23}} \big) \varphi_{t}^{12} - E\big(t_{22}(1 - |t_{33}|^{2}) +t_{23}t_{32} \overline{t_{33}} \big)  \varphi_{t}^{1\overline{2}} \label{eq:phi12}\\
    & + \big(t_{23}\overline{t_{33}} + E^{2} t_{22}\overline{t_{23}}\big) \varphi_{t}^{13} - \overline{E} (t_{23} - E^{2}t_{32}|t_{23}|^{2} + E^{2} t_{22}t_{33} \overline{t_{23}} )  \varphi_{t}^{1\overline{3}} \Big), \notag \\
    \varphi^{13} = &  \frac{1}{f} \Big(- t_{31} \varphi_{t}^{1\overline{1}} + (t_{32} \overline{t_{22}} + \overline{E}^{2} t_{33}\overline{t_{32}}) \varphi_{t}^{12} - E (t_{32} - \overline{E}^{2} t_{23} |t_{32}|^{2} + \overline{E}^{2}t_{22}t_{33} \overline{t_{32}} ) \varphi_{t}^{1\overline{2}} \label{eq:phi13} \\
    & + \big(1 - |t_{22}|^{2} - \overline{E}^{2}t_{23} \overline{t_{32}} \big) \varphi^{13}_{t} - \overline{E} \big(t_{33}(1 - |t_{22}|^{2}) + t_{23}t_{32} \overline{t_{22}}  \big)\varphi_{t}^{1\overline{3}}  \Big), \notag\\
    \varphi^{1 \overline{2}} = & \frac{1}{f} \Big( \overline{E} \big(t_{21} \overline{t_{22}} + E^{2} t_{31} \overline{t_{23}} \big) \varphi_{t}^{1 \overline{1}} - \overline{E} \big(\overline{t_{22}} (1 - |t_{33}|^{2}) + t_{33} \overline{t_{23}} \overline{t_{32}} \big)\varphi_{t}^{12} + \big(1 - |t_{33}|^{2} - \overline{E}^{2} t_{23} \overline{t_{32}} \big) \varphi_{t}^{1 \overline{2}}   \label{eq:phi1 2bar} \\
    &  - E \big( \overline{t_{23}} - \overline{E}^{2} \overline{t_{32}}|t_{23}|^{2} + \overline{E}^{2} t_{23}\overline{t_{22}}\overline{t_{33}} \big) \varphi_{t}^{13} + \big(t_{33} \overline{t_{23}} + \overline{E}^{2} t_{23} \overline{t_{22}}  \big) \varphi_{t}^{1 \overline{3}}    \Big), \notag \\
    \varphi^{1 \overline{3}} = & \frac{1}{f} \Big( \overline{E} \big(t_{21} \overline{t_{32}} + E^{2}t_{31} \overline{t_{33}} \big) \varphi_{t}^{1 \overline{1}} - \overline{E} \big( \overline{t_{32}} - E^{2} \overline{t_{23}} |t_{32}|^{2} + E^{2} t_{32} \overline{t_{22}} \overline{t_{33}} \big) \varphi_{t}^{1 2} + \big(t_{22} \overline{t_{32}} + E^{2} t_{32} \overline{t_{33}} \big)\varphi_{t}^{1 \overline{2}}  \label{eq:phi1 3bar}\\
    &   - E  \big(\overline{t_{33}} (1 - |t_{22}|^{2}) +t_{22} \overline{t_{23}} \overline{t_{32}} \big) \varphi_{t}^{1 3} + \big(1 - |t_{22}|^{2} - E^{2} t_{32} \overline{t_{23}}  \big) \varphi_{t}^{1 \overline{3}} \Big). \notag
\end{align}

\bibliography{Special_structures}

\begin{bibdiv}
\begin{biblist}

\bib{AbbenaGrassi86}{article}{
      author={Abbena, E.},
      author={Grassi, A.},
       title={Hermitian left invariant metrics on complex {L}ie groups and
  cosymplectic {H}ermitian manifolds},
        date={1986},
     journal={Boll. Un. Mat. Ital. A (6)},
      volume={5},
      number={3},
       pages={371\ndash 379},
}

\bib{Alessandrini17}{article}{
      author={Alessandrini, L.},
       title={Proper modifications of generalized {$p$}-{K}\"ahler manifolds},
        date={2017},
     journal={J. Geom. Anal.},
      volume={27},
      number={2},
       pages={947\ndash 967},
}

\bib{AlessandriniBassanelli90}{article}{
      author={Alessandrini, L.},
      author={Bassanelli, G.},
       title={Small deformations of a class of compact non-{K}\"ahler
  manifolds},
        date={1990},
     journal={Proc. Amer. Math. Soc.},
      volume={109},
      number={4},
       pages={1059\ndash 1062},
}

\bib{AlessandriniBassanelli91modifications}{incollection}{
      author={Alessandrini, L.},
      author={Bassanelli, G.},
       title={Smooth proper modifications of compact {K}\"ahler manifolds},
        date={1991},
   booktitle={Complex analysis ({W}uppertal, 1991)},
      series={Aspects Math.},
      volume={E17},
   publisher={Friedr. Vieweg, Braunschweig},
       pages={1\ndash 7},
}

\bib{AlexandrovIvanov2001}{article}{
      author={Alexandrov, B.},
      author={Ivanov, S.},
       title={Vanishing theorems on {H}ermitian manifolds},
        date={2001},
     journal={Differential Geom. Appl.},
      volume={14},
      number={3},
       pages={251\ndash 265},
}

\bib{Angella13}{article}{
      author={Angella, D.},
       title={The cohomologies of the {I}wasawa manifold and of its small
  deformations},
        date={2013-12},
     journal={The Journal of Geometric Analysis},
      volume={23},
       pages={1355–1378},
}

\bib{AngellaFranziniRossi}{article}{
      author={Angella, D.},
      author={Franzini, M.~G.},
      author={Rossi, F.~A.},
       title={Degree of non-{K}\"ahlerianity for 6-dimensional nilmanifolds},
        date={2015},
     journal={Manuscripta Math.},
      volume={148},
      number={1-2},
       pages={177\ndash 211},
}

\bib{AngellaGuedjLu}{article}{
      author={Angella, D.},
      author={Guedj, V.},
      author={Lu, C.~H.},
       title={Plurisigned {H}ermitian metrics},
        date={2023},
     journal={Trans. Amer. Math. Soc.},
      volume={376},
      number={7},
       pages={4631\ndash 4659},
}

\bib{AngellaTomassini11}{article}{
      author={Angella, D.},
      author={Tomassini, A.},
       title={On cohomological decomposition of almost-complex manifolds and
  deformations},
        date={2011},
        ISSN={0926-2245},
     journal={Journal of Symplectic Geometry},
      volume={9},
       pages={403\ndash 428},
}

\bib{AngellaTomassini13}{article}{
      author={Angella, D.},
      author={Tomassini, A.},
       title={On the {$\partial\overline{\partial}$}-lemma and {B}ott-{C}hern
  cohomology},
        date={2013},
     journal={Invent. Math.},
      volume={192},
      number={1},
       pages={71\ndash 81},
}

\bib{AngellaUgarte17}{article}{
      author={Angella, D.},
      author={Ugarte, L.},
       title={On small deformations of balanced manifolds},
        date={2017},
     journal={Differential Geom. Appl.},
      volume={54},
       pages={464\ndash 474},
}

\bib{Bismut89}{article}{
      author={Bismut, J.-M.},
       title={A local index theorem for non-{K}\"ahler manifolds},
        date={1989},
     journal={Math. Ann.},
      volume={284},
      number={4},
       pages={681\ndash 699},
}

\bib{Ciulica25}{article}{
      author={Ciulică, C.},
       title={A new class of non-{K}\"ahler metrics},
        date={2025},
     journal={Complex Manifolds},
      volume={12},
      number={1},
       pages={Paper No. 20250014, 14},
}

\bib{CiulicaOtimanStanciu}{article}{
      author={Ciulică, C.},
      author={Otiman, A.},
      author={Stanciu, M.},
       title={Special non-{K}\"ahler metrics on {E}ndo-{P}ajitnov manifolds},
        date={2025},
     journal={Ann. Mat. Pura Appl. (4)},
      volume={204},
      number={4},
       pages={1425\ndash 1441},
}

\bib{CorderoFernandezGrayUgarte97}{article}{
      author={Cordero, L.},
      author={Fernández, M.},
      author={Gray, A.},
      author={Ugarte, L.},
       title={Nilpotent complex structures on compact nilmanifolds},
        date={1997},
     journal={Rend. Circ. Mat. Palermo (2) Suppl.},
      volume={49},
       pages={83\ndash 100},
}

\bib{CorderoFernandezGrayUgarte00}{article}{
      author={Cordero, L.~A.},
      author={Fern\'andez, M.},
      author={Gray, A.},
      author={Ugarte, L.},
       title={Compact nilmanifolds with nilpotent complex structures:
  {D}olbeault cohomology},
        date={2000},
     journal={Trans. Amer. Math. Soc.},
      volume={352},
      number={12},
       pages={5405\ndash 5433},
}

\bib{Ehresmann47}{article}{
      author={Ehresmann, C.},
       title={Sur les espaces fibr\'es diff\'erentiables},
        date={1947},
        ISSN={0001-4036},
     journal={C. R. Acad. Sci. Paris},
      volume={224},
       pages={1611\ndash 1612},
}

\bib{FinoGrantcharovVezzoni19}{article}{
      author={Fino, A.},
      author={Grantcharov, G.},
      author={Vezzoni, L.},
       title={Astheno-{K}\"ahler and balanced structures on fibrations},
        date={2019},
     journal={Int. Math. Res. Not. IMRN},
      number={22},
}

\bib{FinoKasuyaVezzoni15}{article}{
      author={Fino, A.},
      author={Kasuya, H.},
      author={Vezzoni, L.},
       title={S{KT} and tamed symplectic structures on solvmanifolds},
        date={2015},
     journal={Tohoku Math. J. (2)},
      volume={67},
      number={1},
       pages={19\ndash 37},
}

\bib{FinoPartonSalamon04}{article}{
      author={Fino, A.},
      author={Parton, M.},
      author={Salamon, S.},
       title={Families of strong {KT} structures in six dimensions},
        date={2004},
     journal={Comment. Math. Helv.},
      volume={79},
      number={2},
       pages={317\ndash 340},
}

\bib{FinoTomassini09}{article}{
      author={Fino, A.},
      author={Tomassini, A.},
       title={Blow-ups and resolutions of strong {K}\"ahler with torsion
  metrics},
        date={2009},
     journal={Adv. Math.},
      volume={221},
      number={3},
       pages={914\ndash 935},
}

\bib{FinoTomassini11}{article}{
      author={Fino, A.},
      author={Tomassini, A.},
       title={On astheno-{K}\"ahler metrics},
        date={2011},
     journal={J. Lond. Math. Soc. (2)},
      volume={83},
      number={2},
       pages={290\ndash 308},
}

\bib{FinoVezzoni15}{article}{
      author={Fino, A.},
      author={Vezzoni, L.},
       title={Special {H}ermitian metrics on compact solvmanifolds},
        date={2015},
     journal={J. Geom. Phys.},
      volume={91},
       pages={40\ndash 53},
}

\bib{FinoVezzoni2016}{article}{
      author={Fino, A.},
      author={Vezzoni, L.},
       title={On the existence of balanced and {SKT} metrics on nilmanifolds},
        date={2016},
     journal={Proc. Amer. Math. Soc.},
      volume={144},
      number={6},
       pages={2455\ndash 2459},
}

\bib{Franzini11}{misc}{
      author={Franzini, M.~G.},
       title={Deformazioni di varietà bilanciate e loro proprietà
  coomologiche},
        type={Tesi di Laurea Magistrale},
        date={2011},
        note={Tesi di Laurea Magistrale, Università di Parma},
}

\bib{Gauduchon77-fibres}{article}{
      author={Gauduchon, P.},
       title={Fibr\'es hermitiens \`a{} endomorphisme de {R}icci non
  n\'egatif},
        date={1977},
     journal={Bull. Soc. Math. France},
      volume={105},
      number={2},
       pages={113\ndash 140},
}

\bib{Gauduchon84}{article}{
      author={Gauduchon, P.},
       title={La {$1$}-forme de torsion d'une vari\'et\'e{} hermitienne
  compacte},
        date={1984},
     journal={Math. Ann.},
      volume={267},
      number={4},
       pages={495\ndash 518},
}

\bib{Gauduchon97}{article}{
      author={Gauduchon, P.},
       title={Hermitian connections and {D}irac operators},
        date={1997},
     journal={Boll. Un. Mat. Ital. B (7)},
      volume={11},
      number={2},
       pages={257\ndash 288},
}

\bib{GotoGeneralizedKahler}{article}{
      author={Goto, R.},
       title={Deformations of generalized complex and generalized {K}\"ahler
  structures},
        date={2010},
     journal={J. Differential Geom.},
      volume={84},
      number={3},
       pages={525\ndash 560},
}

\bib{GrayHervella80}{article}{
      author={Gray, A.},
      author={Hervella, L.~M.},
       title={The sixteen classes of almost {H}ermitian manifolds and their
  linear invariants},
        date={1980},
     journal={Ann. Mat. Pura Appl. (4)},
      volume={123},
       pages={35\ndash 58},
}

\bib{GuanLi}{article}{
      author={Guan, B.},
      author={Li, Q.},
       title={Complex {M}onge-{A}mp\`ere equations and totally real
  submanifolds},
        date={2010},
     journal={Adv. Math.},
      volume={225},
      number={3},
       pages={1185\ndash 1223},
}

\bib{HarveyKnapp74}{incollection}{
      author={Harvey, R.},
      author={Knapp, A.~W.},
       title={Positive {$(p,\,p)$} forms, {W}irtinger' inequality, and
  currents},
        date={1974},
   booktitle={Value distribution theory ({P}roc. {T}ulane {U}niv. {P}rogram,
  {T}ulane {U}niv., {N}ew {O}rleans, {L}a., 1972-1973), {P}art {A}},
      series={Pure Appl. Math.},
      volume={25},
   publisher={Dekker, New York},
       pages={43\ndash 62},
}

\bib{Has10}{article}{
      author={Hasegawa, K.},
       title={Small deformations and non-left-invariant complex structures on
  six-dimensional compact solvmanifolds},
        date={2010},
     journal={Differential Geom. Appl.},
      volume={28},
      number={2},
       pages={220\ndash 227},
}

\bib{HindMedoriTomassini23}{article}{
      author={Hind, R.},
      author={Medori, C.},
      author={Tomassini, A.},
       title={Families of almost complex structures and transverse
  {$(p,p)$}-forms},
        date={2023},
     journal={J. Geom. Anal.},
      volume={33},
      number={10},
       pages={Paper No. 334, 23},
}

\bib{Huybrechts05}{book}{
      author={Huybrechts, D.},
       title={Complex {G}eometry},
    subtitle={An {I}ntroduction},
      series={Universitext},
   publisher={Springer Berlin Heidelberg},
        date={2005},
}

\bib{JostYau93}{article}{
      author={Jost, J.},
      author={Yau, S.-T.},
       title={A nonlinear elliptic system for maps from {H}ermitian to
  {R}iemannian manifolds and rigidity theorems in {H}ermitian geometry},
        date={1993},
     journal={Acta Math.},
      volume={170},
      number={2},
       pages={221\ndash 254},
}

\bib{Kasuya14}{article}{
      author={Kasuya, H.},
       title={de {R}ham and {D}olbeault cohomology of solvmanifolds with local
  systems},
        date={2014},
     journal={Math. Res. Lett.},
      volume={21},
      number={4},
       pages={781\ndash 805},
}

\bib{Kasuya17}{article}{
      author={Kasuya, H.},
       title={Generalized deformations and holomorphic {P}oisson cohomology of
  solvmanifolds},
        date={2017},
     journal={Ann. Global Anal. Geom.},
      volume={51},
      number={2},
       pages={155\ndash 177},
}

\bib{Kodaira05}{book}{
      author={Kodaira, K.},
       title={Complex {M}anifolds and {D}eformations of {C}omplex
  {S}tructures},
      series={Classics in Mathematics},
   publisher={Springer Berlin, Heidelberg},
        date={2005},
}

\bib{KodairaMorrow06}{book}{
      author={Kodaira, K.},
      author={Morrow, J.},
       title={Complex {M}anifolds},
   publisher={AMS Chelsea Publishing},
        date={2006},
         url={https://books.google.it/books?id=eK1N_aPZ0ogC},
}

\bib{KodairaSpencer60}{article}{
      author={Kodaira, K.},
      author={Spencer, D.~C.},
       title={On deformations of complex analytic structures. {III}.
  {S}tability theorems for complex structures},
        date={1960},
     journal={Ann. of Math. (2)},
      volume={71},
       pages={43\ndash 76},
}

\bib{LatorreUgarte2017}{article}{
      author={Latorre, A.},
      author={Ugarte, L.},
       title={On non-{K}\"ahler compact complex manifolds with balanced and
  astheno-{K}\"ahler metrics},
        date={2017},
     journal={C. R. Math. Acad. Sci. Paris},
      volume={355},
      number={1},
       pages={90\ndash 93},
}

\bib{Malcev62}{article}{
      author={Malcev, A.~I.},
       title={On a class of homogeneous spaces},
        date={1962},
     journal={Amer. Math. Soc. Translation Ser. 1},
      volume={9},
       pages={276\ndash 307},
}

\bib{MatsuoTakahashi2001}{article}{
      author={Matsuo, K.},
      author={Takahashi, T.},
       title={On compact astheno-{K}\"ahler manifolds},
        date={2001},
     journal={Colloq. Math.},
      volume={89},
      number={2},
       pages={213\ndash 221},
}

\bib{Michelsohn82}{article}{
      author={Michelsohn, M.~L.},
       title={On the existence of special metrics in complex geometry},
        date={1982},
     journal={Acta Math.},
      volume={149},
      number={3-4},
       pages={261\ndash 295},
}

\bib{Nakamura75}{article}{
      author={Nakamura, I.},
       title={Complex parallelisable manifolds and their small deformations},
        date={1975},
     journal={J. Differential Geometry},
      volume={10},
       pages={85\ndash 112},
}

\bib{PiovaniSferruzza21}{article}{
      author={Piovani, R.},
      author={Sferruzza, T.},
       title={Deformations of {S}trong {K}\"ahler with torsion metrics},
        date={2021},
     journal={Complex Manifolds},
      volume={8},
      number={1},
       pages={286\ndash 301},
}

\bib{Popovici13}{article}{
      author={Popovici, D.},
       title={Deformation limits of projective manifolds: {H}odge numbers and
  strongly {G}auduchon metrics},
        date={2013},
     journal={Invent. Math.},
      volume={194},
      number={3},
       pages={515\ndash 534},
}

\bib{RaoWanZhao18}{article}{
      author={Rao, S.},
      author={Wan, X.},
      author={Zhao, Q.},
       title={On local stabilities of {$p$}-{K}\"ahler structures},
        date={2019},
     journal={Compos. Math.},
      volume={155},
      number={3},
       pages={455\ndash 483},
}

\bib{RaoWanZhao21}{article}{
      author={Rao, S.},
      author={Wan, X.},
      author={Zhao, Q.},
       title={Power series proofs for local stabilities of {K}\"ahler and
  balanced structures with mild {$\partial \overline {\partial }$}-lemma},
        date={2022},
     journal={Nagoya Math. J.},
      volume={246},
       pages={305\ndash 354},
}

\bib{RaoZhao15}{article}{
      author={Rao, S.},
      author={Zhao, Q.},
       title={Extension formulas and deformation invariance of {H}odge
  numbers},
        date={2015},
     journal={C. R. Math. Acad. Sci. Paris},
      volume={353},
      number={11},
       pages={979\ndash 984},
}

\bib{RaoZhao18_2}{article}{
      author={Rao, S.},
      author={Zhao, Q.},
       title={Several special complex structures and their deformation
  properties},
        date={2018},
     journal={J. Geom. Anal.},
      volume={28},
      number={4},
       pages={2984\ndash 3047},
}

\bib{Sferruzza22}{article}{
      author={Sferruzza, T.},
       title={Deformations of balanced metrics},
        date={2022},
     journal={Bull. Sci. Math.},
      volume={178},
       pages={Paper No. 103143, 23},
}

\bib{Sferruzza23}{article}{
      author={Sferruzza, T.},
       title={Deformations of astheno-{K}\"ahler metrics},
        date={2023},
     journal={Complex Manifolds},
      volume={10},
      number={1},
       pages={Paper No. 20230102, 19},
}

\bib{SferruzzaTomassini23}{article}{
      author={Sferruzza, T.},
      author={Tomassini, A.},
       title={On cohomological and formal properties of strong {K}\"ahler with
  torsion and astheno-{K}\"ahler metrics},
        date={2023},
     journal={Math. Z.},
      volume={304},
      number={4},
       pages={Paper No. 55, 27},
}

\bib{SferruzzaTomassini24}{article}{
      author={Sferruzza, T.},
      author={Tomassini, A.},
       title={Bott-{C}hern formality and {M}assey products on strong {K}\"ahler
  with torsion and {K}\"ahler solvmanifolds},
        date={2024},
     journal={J. Geom. Anal.},
      volume={34},
      number={11},
}

\bib{Wang54}{article}{
      author={Wang, H.-C.},
       title={Complex parallisable manifolds},
        date={1954},
     journal={Proceedings of the American Mathematical Society},
      volume={5},
       pages={771\ndash 776},
}

\bib{Wu06}{book}{
      author={Wu, C.-C.},
       title={On the geometry of superstrings with torsion},
   publisher={ProQuest LLC, Ann Arbor, MI},
        date={2006},
        note={Thesis (Ph.D.)--Harvard University},
}

\bib{Xiao15}{article}{
      author={Xiao, J.},
       title={On strongly {G}auduchon metrics of compact complex manifolds},
        date={2015},
        ISSN={1050-6926,1559-002X},
     journal={J. Geom. Anal.},
      volume={25},
      number={3},
       pages={2011\ndash 2027},
}

\bib{Yamada05}{article}{
      author={Yamada, T.},
       title={A construction of compact pseudo-{K}\"ahler solvmanifolds with no
  {K}\"ahler structures},
        date={2005},
     journal={Tsukuba J. Math.},
      volume={29},
      number={1},
       pages={79\ndash 109},
         url={https://doi.org/10.21099/tkbjm/1496164895},
}

\end{biblist}
\end{bibdiv}

\end{document}